\documentclass[reqno]{amsart}
\usepackage{amsfonts,amsmath,amssymb}

\numberwithin{equation}{section}

\newtheorem{thm}{Theorem}[section]
\newtheorem{lem}{Lemma}[section]

\newtheorem{prop}{Proposition}[section]
\newtheorem{defi}{Definition}[section]

\begin{document}
\title[Asymptotic analysis for Paneitz equations]
{Asymptotic analysis for fourth order Paneitz equations with critical growth}
\author{Emmanuel Hebey}
\address{Emmanuel Hebey, Universit\'e de Cergy-Pontoise, 
D\'epartement de Math\'ematiques, Site de 
Saint-Martin, 2 avenue Adolphe Chauvin, 
95302 Cergy-Pontoise cedex, 
France}
\email{Emmanuel.Hebey@math.u-cergy.fr}
\author{Fr\'ed\'eric Robert}
\address{Fr\'ed\'eric Robert, Institut Elie Cartan, Universit\'e Henri Poincar\'e - Nancy 1, 
B.P. 239, 54506 Vandoeuvre-l\`es-Nancy Cedex, France}
\email{frobert@iecn.u-nancy.fr}

\date{February 24, 2010. Revised September 27, 2010.}
\thanks{The authors were partially supported by the ANR grant 
ANR-08-BLAN-0335-01.}

\begin{abstract} We investigate fourth order Paneitz equations of critical growth 
in the case of $n$-dimensional closed conformally flat 
manifolds, $n \ge 5$. Such equations arise from 
conformal geometry and are modelized on the Einstein case of the geometric equation 
describing the effects of conformal changes of metrics on the $Q$-curvature. We obtain sharp 
asymptotics for arbitrary bounded energy sequences of solutions of our equations from 
which we derive stability and compactness properties. In doing so we establish the criticality of the 
geometric equation with respect to the trace of its second order terms.
\end{abstract}

\maketitle

In 1983, Paneitz \cite{Pan} introduced a conformally covariant fourth order operator extending the conformal Laplacian. 
Branson and {\O}rsted \cite{BraOrs}, and Branson \cite{Bra1,Bra2}, introduced the associated notion of $Q$-curvature 
when $n = 4$ and in higher dimensions when dealing with the conformally covariant extensions of the Paneitz 
operator by Graham-Jenne-Mason-Sparling. 
The scalar and the $Q$-curvatures are respectively, up to the conformally invariant Weyl's tensor in dimension four, 
the integrands in dimensions two and four for the 
Gauss-Bonnet formula for the Euler characteristic. The articles by Branson and Gover \cite{BraGov}, Chang \cite{Cha1,Cha2}, 
Chang and Yang \cite{ChaYan}, and Gursky \cite{Gur} contain several references and many interesting  material on the geometric 
and physics aspects associated to this notion of $Q$-curvature. 

\medskip In what follows we let $(M,g)$ be a smooth compact Riemannian manifold 
of dimension $n \ge 5$ and consider the fourth order variational Paneitz equations of critical Sobolev 
growth which are written as
\begin{equation}\label{CritEqt}
\Delta_g^2u + b\Delta_gu + cu = u^{2^\sharp-1}
\hskip.1cm ,
\end{equation}
where $\Delta_g = -\hbox{div}_g\nabla u$ is the Laplace-Beltrami operator, $b, c > 0$ 
are positive real numbers such that $c-\frac{b^2}{4} < 0$, $u$ is required to be positive, and 
$2^\sharp = \frac{2n}{n-4}$ 
is the critical Sobolev exponent. 
Equations like \eqref{CritEqt} are modeled on the conformal equation associated to 
the Paneitz operator when the background metric $g$ is Einstein. In few words, the conformal equation associated to the Paneitz operator, 
relating the $Q$-curvatures $Q_g$ and $Q_{\tilde g}$ of conformal metrics on arbitrary manifolds, is written as
\begin{equation}\label{ConfEqt}
\Delta_g^2u - \hbox{div}_g(A_gdu) + \frac{n-4}{2}Q_gu = \frac{n-4}{2}Q_{\tilde g}u^{2^\sharp-1}
\hskip.1cm ,
\end{equation}
where $\tilde g = u^{4/(n-4)}g$ and, if $Rc_g$ and $S_g$ denote the Ricci and scalar curvature of $g$, 
 $A_g$ is the smooth $(2,0)$-tensor field given by
\begin{equation}\label{DefAg}
A_g = \frac{(n-2)^2+4}{2(n-1)(n-2)}S_gg - \frac{4}{n-2}Rc_g
\hskip.1cm .
\end{equation}
When $g$ is Einstein, so that $Rc_g = \lambda g$ for some $\lambda \in \mathbb{R}$, equation \eqref{ConfEqt} 
can be simplified and written as
\begin{equation}\label{ConfEinstEqt}
\Delta_g^2u + \frac{b_n\lambda}{n-1}\Delta_gu + \frac{c_n\lambda^2}{(n-1)^2}u = \frac{n-4}{2}Q_{\tilde g}u^{2^\sharp-1}
\hskip.1cm ,
\end{equation}
where $b_n$ and $c_n$ are given by
\begin{equation}\label{CoeffGeom}
b_n = \frac{n^2-2n-4}{2}\hskip.1cm\hbox{and}\hskip.1cm c_n = \frac{n(n-4)(n^2-4)}{16}\hskip.1cm .
\end{equation}
In particular, as mentioned, \eqref{CritEqt} is of the type \eqref{ConfEinstEqt}. An important remark is that 
$c_n - \frac{b_n^2}{4} = -1$ is negative. Given $\Lambda > 0$ we let $\mathcal{S}_{b,c}^\Lambda$ be the set
\begin{equation}\label{DefSlice}
\mathcal{S}_{b,c}^\Lambda = \left\{u \in C^4(M), u > 0,~\hbox{s.t.}~\Vert u\Vert_{H^2} \le \Lambda
~\hbox{and}~u~\hbox{solves}~\eqref{CritEqt}\right\}\hskip.1cm ,
\end{equation}
where $H^2$ is the Sobolev space 
of functions in $L^2$ with two derivatives in $L^2$. Following standard terminology, we say that 
\eqref{CritEqt} is {\it compact} if for any $\Lambda > 0$, $\mathcal{S}_{b,c}^\Lambda$ is compact in the 
$C^4$-topology (we adopt here the bounded version of compactness, as first introduced by Schoen \cite{Sch}). 
The stronger notion of stability we discuss in the sequel is defined as follows:

\begin{defi}\label{DefStab} Equation \eqref{CritEqt} is {\it stable} if it is compact and if for 
any $\Lambda > 0$, and any $\varepsilon > 0$, there exists 
$\delta >  0$ such that for any $b^\prime$ and $c^\prime$, it holds that 
\begin{equation}\label{GeomStab}
d_{C^4}^{\hookrightarrow}(\mathcal{S}_{b^\prime,c^\prime}^\Lambda;\mathcal{S}_{b,c}^\Lambda) < \varepsilon
\end{equation}
as soon as $\vert b^\prime - b\vert +  \vert c^\prime - c\vert < \delta$, where $\mathcal{S}_{b,c}^\Lambda$ and $\mathcal{S}_{b^\prime,c^\prime}^\Lambda$ 
are given by \eqref{DefSlice}, and 
where for $X, Y \subset C^4$, $d_{C^4}^{\hookrightarrow}(X;Y)$ is the pointed distance defined as the sup over 
the $u \in X$ of the inf over the $v \in Y$ of $\Vert v-u\Vert_{C^4}$.
\end{defi}

The meaning of \eqref{GeomStab} is that small perturbations of $b$ and $c$ 
in \eqref{CritEqt} do not create 
solutions which stand far from solutions of the original equation. Stability is an important notion in view of topological arguments
and degree theory. Also it has a natural translation in terms  of phase stability for solitons
of the fourth order Schr\"odinger equations
introduced by Karpman \cite{Kar} and Karpman and Shagalov \cite{KarSha} (see the remark 
at the end of Section \ref{ProofTheorem2Part2}). The main questions we ask here are:

\medskip\noindent{\bf Questions:} (Q1) {\it describe and control the asymptotic behavior of arbitrary 
finite energy sequences of solutions of equations like \eqref{CritEqt}.}

\noindent (Q2) {\it find conditions on $b$ and $c$ for \eqref{CritEqt} to be stable.}

\medskip\noindent By contradiction, \eqref{CritEqt} 
is stable if and only if for any sequences $(b_\alpha)_\alpha$ 
and $(c_\alpha)_\alpha$ of real numbers converging to $b$ and $c$, and any sequence $(u_\alpha)_\alpha$ of smooth positive 
solutions of
\begin{equation}\label{PertCritEqt}
\Delta_g^2u + b_\alpha\Delta_gu + c_\alpha u = u^{2^\sharp-1}
\end{equation}
such that $(u_\alpha)_\alpha$ is bounded in $H^2$, there holds that, up to a subsequence, 
$u_\alpha \to u_\infty$ in $C^4(M)$, where $u_\infty$ is a smooth positive solution of \eqref{CritEqt}. 
In other words, \eqref{CritEqt} is stable if we can impede bubbling for arbitrary bounded sequences in $H^2$ of 
solutions of arbitrary sequences of equations like \eqref{PertCritEqt}, including \eqref{CritEqt} itself. In order to do so, we need 
sharp answers to (Q1).

\medskip As is well known, critical equations tend to be unstable (precisely because of 
the bubbling which is usually associated with critical equations). A consequence of 
Theorem \ref{StabilityThm} below is that bubbling is not only associated with the criticality of the equation but also with the 
geometry through the relation $b = \frac{1}{n}\hbox{Tr}_g(A_g)$ 
which, see \eqref{AddedEqtIntro} below, characterizes the middle term of the geometric equation \eqref{ConfEinstEqt}. 

\medskip Concerning the bound on the energy we require in Definition \ref{DefStab}, it should be noted that we cannot expect 
the existence of a priori $H^2$-bounds for arbitrary sequences of equations like \eqref{PertCritEqt} 
when dealing with large coefficents $b$ and $c$ (like it is the case for Yamabe type equations associated with second order Schr\"odinger operators 
with large potentials). In parallel it is intuitively clear that bounded sequences in $H^2$ of solutions of equations like 
\eqref{PertCritEqt} can develop an arbitrarily large number of peaks. Summing sphere singularities 
in a naive way we indeed can prove, see Hebey, Robert and Wen \cite{HebRobWen}, 
that for any quotient of the $n$-sphere, $n \ge 12$, there exist sequences $(u_\alpha)_\alpha$ and $(v_\alpha)_\alpha$ of smooth positive solutions of
$$\Delta_g^2u + b_n\Delta_gu + c_\alpha u = u^{2^\sharp-1}$$
such that $(u_\alpha)_\alpha$ blows up with an arbitrarily large given number $k$ of peaks and $\Vert v_\alpha\Vert_{H^2} \to +\infty$ 
as $\alpha \to +\infty$, 
where $(c_\alpha)_\alpha$ is a sequence of smooth functions converging in the $C^1$-topology to $c_n$, and $b_n$ and $c_n$ are as in 
\eqref{CoeffGeom}. In other words, illustrating the above discussion, we see that equations like \eqref{CritEqt} create bubbling, 
even multiple of cluster type (namely with blow-up points collapsing 
on a single point), and that there is no 
statement about universal a priori $H^2$-bounds for arbitrary solutions of arbitrary equations like \eqref{PertCritEqt}. Also we see that an 
equation can be compact and unstable (the geometric equation is
compact on quotients of the sphere). 
Compactness for the geometric equation in the conformally flat 
case has been established by Hebey and Robert \cite{HebRob}, and by 
Qing and Raske \cite{QinRas1,QinRas2}. 
The elegant geometric approach in Qing and Raske \cite{QinRas1,QinRas2} 
is based on the integral representation of the solutions 
through the developing map under 
the natural assumption that the Poincar\'e exponent is small. Recently, 
Wei and Zhao \cite{WeiZha} constructed blow-up examples 
in the non conformally flat case when $n \ge 25$.

\medskip Let $\lambda_i(A_g)_x$, 
$i = 1,\dots,n$, be the $g$-eigenvalues of $A_g(x)$ repeated with their multiplicity. Let $\lambda_1$ be the infimum over $i$ and $x$,
 and $\lambda_2$ be the supremum over $i$ and $x$ of the $\lambda_i(A_g)_x$'s. Following Hebey, Robert and Wen \cite{HebRobWen} 
 we define the wild spectrum of $A_g$ to be the interval $\mathcal{S}_w = [\lambda_1,\lambda_2]$. 
It was proved 
in \cite{HebRobWen} that \eqref{CritEqt} is stable on conformally flat manifolds when $n = 6,7, 8$ and 
$b < \lambda_1$, or $n \ge 9$ and $b \not\in \mathcal{S}_w$. We improve these results in different important significative directions in the 
present article: we add the case of dimension $n = 5$,
we replace the condition $b \not\in \mathcal{S}_w$ 
by the much weaker condition $b \not= \frac{1}{n}\hbox{Tr}_g(A_g)$, and 
we accept large values of $b$ when $n = 6, 7$. On the other hand, we leave open the question of getting similar results 
in the nonconformally flat case. 
In the above discussion, and in what follows, $\hbox{Tr}_g(A_g) = g^{ij}A_{ij}$ is the trace of $A_g$ with respect to $g$. There clearly holds that 
$\frac{1}{n}\hbox{Tr}_g(A_g) \in \mathcal{S}_w$ at any point in $M$, and it is easily seen that
\begin{equation}\label{AddedEqtIntro}
\hbox{Tr}_g(A_g) = \frac{n^2-2n-4}{2(n-1)}S_g\hskip.1cm .
\end{equation}
Let $(u_\alpha)_\alpha$ be a bounded sequence in $H^2$ of solutions of \eqref{PertCritEqt}. Up to a subsequence, 
$u_\alpha \rightharpoonup u_\infty$ weakly in $H^2$ for some $u_\infty \in H^2$ which solves \eqref{CritEqt}. When $c - \frac{b^2}{4} < 0$, by 
the maximum principle, either $u_\infty > 0$ in $M$ or $u_\infty \equiv 0$. 
In the second order case, in low dimensions 
(namely $n = 3, 4, 5$) we know from Druet \cite{Dru} that we necessarily have that $u_\infty \equiv 0$ if the convergence of $u_\alpha$ to 
$u_\infty$ is not strong (but only weak) and the $u_\alpha$'s solve Yamabe type equations. 
In the framework of question (Q1), we also address in this article the question of whether or not such type of results extend 
to the fourth order case when passing from Yamabe type equations to Paneitz equations like \eqref{CritEqt}. We positively answer 
to this question in Theorem \ref{LimitProfileThm} below, the low dimensions being now $5, 6, 7$. 

\begin{thm}\label{LimitProfileThm} Let $(M,g)$ be a smooth compact conformally flat Riemannian manifold of dimension $n = 5, 6, 7$ 
and $b, c > 0$ be positive real numbers such that $c - \frac{b^2}{4} < 0$. Let $(b_\alpha)_\alpha$ 
and $(c_\alpha)_\alpha$ be sequences of real numbers converging to $b$ and $c$, and $(u_\alpha)_\alpha$ be a bounded sequence 
in $H^2$ of smooth positive solutions of \eqref{PertCritEqt} such that $u_\alpha \rightharpoonup u_\infty$ weakly in $H^2$ 
as $\alpha \to +\infty$. Then either $u_\alpha \to u_\infty$ strongly in any $C^k$-topology, or $u_\infty \equiv 0$.
\end{thm}

Theorem \ref{LimitProfileThm} answers the above mentioned question of whether or not we can have a nontrivial limit profile for blowing-up sequences 
of solutions of \eqref{PertCritEqt}. As a remark, the geometric equation on the sphere provides in any dimension $n \ge 5$ 
an example of an equation like \eqref{CritEqt} with sequences 
$(u_\alpha)_\alpha$ of solutions such that $u_\alpha \not\to u_\infty$ strongly and $u_\infty \equiv 0$. 
Now we return to the question of the stability of \eqref{CritEqt}. When $n = 5$ we let $G$ be the Green's function 
of the fourth order Paneitz type operator $P_g = \Delta_g^2 + b\Delta_g + c$. Then
\begin{equation}\label{GreenFct}
G_x(y) = \frac{1}{6\omega_4d_g(x,y)} + \mu_x(y)
\hskip.1cm ,
\end{equation}
where $G_x(\cdot) = G(x,\cdot)$ is the Green's function at $x$ of $P_g$, $\omega_4$ is the volume of the unit 
$5$-sphere, and $\mu_x$ is $C^{0,\theta}$ in $M$ for $\theta \in (0,1)$. The mass at $x$ of $P_g$ is $\mu_x(x)$. Our second result states as follows.

\begin{thm}\label{StabilityThm} Let $(M,g)$ be a smooth compact conformally flat Riemannian manifold of dimension $n \ge 5$ 
and $b, c > 0$ be positive real numbers such that $c - \frac{b^2}{4} < 0$. Assume that one of the following conditions holds true:

(i) $n = 5$ and $\mu_x(x) > 0$ for all $x$\hskip.1cm ,

(ii) $n = 6$ and $b \not\in \mathcal{S}_w$\hskip.1cm ,

(iii) $n = 8$ and $b < \frac{1}{8}\min_M\hbox{Tr}_g(A_g)$\hskip.1cm,

(iv) $n = 7$ or $n \ge 9$ and $b \not= \frac{1}{n}\hbox{Tr}_g(A_g)$ in $M$\hskip.1cm .

\noindent Then for any sequences $(b_\alpha)_\alpha$ 
and $(c_\alpha)_\alpha$ of real numbers converging to $b$ and $c$, and any bounded sequence $(u_\alpha)_\alpha$ in $H^2$ of smooth positive 
solutions of \eqref{PertCritEqt} there holds that, up to a subsequence, 
$u_\alpha \to u_\infty$ in $C^4(M)$ for some smooth positive solution $u_\infty$ of \eqref{CritEqt}. In particular,  \eqref{CritEqt} is stable.
\end{thm}

As a direct consequence of Theorem \ref{StabilityThm}, cluster solutions and bubble towers do not exist for \eqref{CritEqt} when one of the 
conditions (i) to (iv) is assumed to hold. This includes the existence of cluster solutions or bubble towers constructed by means of perturbing \eqref{CritEqt} 
such as in \eqref{PertCritEqt}. 

\medskip Let $G_0$ be the Green's function of the geometric Paneitz operator $P_0$ in the left hand side of \eqref{ConfEqt}. Humbert and Raulot \cite{HumRau} proved 
the very nice result that in the conformally flat case, assuming that the Yamabe invariant is positive, that $P_0$ is positive, and that $G_0 > 0$ outside 
the diagonal, then the mass 
of $G_0$ is nonnegative and equal to zero at one point 
if and only if the manifold is conformally diffeomorphic to the sphere. A similar result was previously established by Qing and Raske \cite{QinRas1,QinRas2} 
when the Poincar\'e exponent is small. By (i) in Theorem \ref{StabilityThm} we need to find conditions under which $\mu_x(x) > 0$ 
for all $x$, where $\mu_x(x)$ is the mass of our operator $P_g = \Delta_g^2 + b\Delta_g + c$. A 
third theorem we prove, based on the Humbert and Raulot \cite{HumRau} result, is as follows.

\begin{thm}\label{PositMass} Let $(M,g)$ be a smooth compact conformally flat Riemannian manifold of dimension $n=5$ with positive Yamabe invariant such that the Green's function of the geometric Paneitz operator $P_0$ is positive and let $b,c>0$ be positive real numbers. We assume that
$b g\leq A_g$ in the sense of bilinear forms and $c\leq\frac{1}{2}Q_g$, 
and in case $A_g\equiv bg$ and $c\equiv\frac{1}{2}Q_g$ simultaneously, we assume in addition that $(M,g)$ is not conformally diffeomorphic to the standard sphere. Then the mass $\mu_x(x)$ of $P_g$ is positive 
for all $x \in M$.
\end{thm}

Theorem \ref{PositMass} raises the question of the existence of $5$-dimensional compact conformally flat manifolds with 
positive Yamabe invariant and positive Green's function for $P_0$. By the analysis in Qing and Raske \cite{QinRas2} compact conformally flat
manifolds of positive Yamabe invariant and of Poincar\'e exponent less that $\frac{n-4}{2}=\frac{1}{2}$ are such that the  
Green's function of $P_0$ is positive. Explicit
examples of such manifolds are in Qing-Raske \cite{QinRas2}.

\medskip The paper is organized as follows. In Section \ref{PtEst} we establish sharp pointwise estimates 
for arbitrary sequences of solutions of \eqref{PertCritEqt}. 
This answers (Q1). 
Thanks to these estimates we prove Theorem \ref{LimitProfileThm} in Section \ref{ProofThm1}. Trace estimates are proved to hold in Section \ref{ProofTraceEst}. 
By the estimates in Sections \ref{PtEst} and \ref{ProofTraceEst} we can prove Theorem \ref{StabilityThm} in Section \ref{ProofTheorem2Part1} when $n \ge 6$. 
In Section \ref{ProofTheorem2Part2} we prove Theorem 
\ref{StabilityThm} in the specific case $n = 5$. Theorem \ref{StabilityThm} provides the answer to (Q2). 
We prove Theorem \ref{PositMass} in Section \ref{ProofPosMassThm}. 

\section{Pointwise estimates}\label{PtEst}

Let $(M,g)$ be a smooth compact Riemannian manifold of dimension $n \ge 5$, and $b, c > 0$ be positive real 
numbers such that $c - \frac{b^2}{4} < 0$. We do not need in this section to assume that 
$g$ is conformally flat. Let also $(b_\alpha)_\alpha$ and $(c_\alpha)_\alpha$ be converging sequences 
of real numbers with limits $b$ and $c$ as $\alpha \to \infty$, and $(u_\alpha)_\alpha$ be a bounded sequence in $H^2$ of 
positive nontrivial solutions of \eqref{PertCritEqt}. Up to a subsequence, $u_\alpha \rightharpoonup u_\infty$ weakly in $H^2$ 
as $\alpha \to +\infty$. By standard elliptic theory, either $u_\alpha \to u_\infty$ in $C^4$ or the $u_\alpha$'s blow up and
\begin{equation}\label{BlowUpAssumpt}
\Vert u_\alpha\Vert_{L^\infty} \to +\infty
\end{equation}
as $\alpha \to +\infty$. From now on we assume that \eqref{BlowUpAssumpt} holds true. 
By Hebey and Robert \cite{HebRob0} and Hebey, Robert and Wen \cite{HebRobWen}, there 
holds that
\begin{equation}\label{EqtBlowUp1}
u_\alpha = u_\infty + \sum_{i=1}^kB_\alpha^i + \mathcal{R}_\alpha\hskip.1cm ,
\end{equation}
where $\mathcal{R}_\alpha \to 0$ in $H^2$ as $\alpha \to \infty$, and the $B_\alpha^i$'s are bubble singularities in $H^2$. Such 
$B_\alpha^i$'s are given by
\begin{equation}\label{EqtBlowUp2}
B_\alpha^i = \eta\left(d_{i,\alpha}\right) \left(\frac{\mu_{i,\alpha}}{\mu_{i,\alpha}^2 + \frac{d_\alpha^2}{\sqrt{\lambda_n}}}\right)^{\frac{n-4}{2}}
\hskip.1cm ,
\end{equation}
where $\eta: \mathbb{R} \to \mathbb{R}$ is a smooth nonnegative cutoff function with small support (less than the injectivity 
radius of $g$), $d_{i,\alpha}(\cdot) = d_g(x_{i,\alpha},\cdot)$, $\lambda_n = n(n-4)(n^2-4)$, $k \ge 1$ is an integer, and for any $i$, $(x_{i,\alpha})_\alpha$ is 
a converging sequence of points in $M$ and $(\mu_{i,\alpha})_\alpha$ is a sequence of positive real numbers such that $\mu_{i,\alpha} \to 0$ 
as $\alpha \to +\infty$. Moreover, we also have that the following structure equation holds true: for any $i\not= j$,
\begin{equation}\label{EqtBlowUp3}
\frac{d_g(x_{i,\alpha},x_{j,\alpha})^2}{\mu_{i,\alpha}\mu_{j,\alpha}} + \frac{\mu_{i,\alpha}}{\mu_{j,\alpha}} + 
\frac{\mu_{i,\alpha}}{\mu_{j,\alpha}} \to +\infty
\end{equation}
as $\alpha \to +\infty$, and that there exists $C > 0$ such that
\begin{equation}\label{EqtBlowUp4}
r_\alpha(x)^{\frac{n-4}{2}}\left\vert u_\alpha(x) - u_\infty(x)\right\vert \le C
\end{equation}
for all $\alpha$ and all $x \in M$, where 
\begin{equation}\label{DefrAlpha}
r_\alpha(x) = \min_{i=1,\dots,k}d_g(x_{i,\alpha},x)
\hskip.1cm .
\end{equation}
By \eqref{EqtBlowUp1} and \eqref{EqtBlowUp4} we have that $u_\alpha \rightharpoonup u_\infty$ in $H^2$ and $u_\alpha \to u_\infty$ in 
$C^0_{loc}(M\backslash\mathcal{S})$, where $\mathcal{S}$ is the set consisting of the limits of the $x_{i,\alpha}$'s. 
We define $\mu_\alpha$ by
\begin{equation}\label{DefMuAlpha}
\mu_\alpha = \max_{i=1,\dots,k}\mu_{i,\alpha}
\hskip.1cm .
\end{equation}
There holds that $\mu_\alpha \to 0$ as $\alpha \to +\infty$ since $\mu_{i,\alpha} \to 0$ for all $i$ as 
$\alpha \to +\infty$. 
We aim here at proving the following sharp pointwise estimates.

\begin{prop}\label{SharpPtEst} Let $(M,g)$ be a smooth compact Riemannian manifold of dimension $n \ge 5$, and $b, c > 0$ be positive real 
numbers such that $c - \frac{b^2}{4} < 0$. Let also $(b_\alpha)_\alpha$ and $(c_\alpha)_\alpha$ be converging sequences 
of real numbers with limits $b$ and $c$ as $\alpha \to +\infty$, and $(u_\alpha)_\alpha$ be a bounded sequence in $H^2$ of 
positive nontrivial solutions of \eqref{PertCritEqt} satisfying \eqref{BlowUpAssumpt}. There exists $C > 0$ such that, up to a subsequence,
\begin{equation}\label{EqtMnPrp}
\left\vert\nabla^ju_\alpha\right\vert \le C\left(\mu_\alpha^{\frac{n-4}{2}}r_\alpha^{4-n-j} + \Vert u_\infty\Vert_{L^\infty}\right)
\end{equation}
in $M$, for all $j=0,1,2,3$ and all $\alpha$, where $r_\alpha$ is as in \eqref{DefrAlpha}, and $\mu_\alpha$ is as in \eqref{DefMuAlpha}.
\end{prop}

We split the proof of Proposition \ref{SharpPtEst} into several lemmas. The first lemma we prove 
is a general basic result we will use further in the proof of Lemma \ref{Lem2}.

\begin{lem}\label{Lem1} Let $\delta > 0$ and $(g_\alpha)_\alpha$ be a sequence of Riemannian metrics in the Euclidean ball 
$B_0(2\delta)$ such that $g_\alpha \to \xi$ in $C^4_{loc}\left(B_0(2\delta)\right)$ as $\alpha \to +\infty$, where 
$\xi$ is the Euclidean metric. Let $(b_\alpha)_\alpha$ and $(c_\alpha)_\alpha$ be bounded sequences of positive real numbers such that 
$c_\alpha \le \frac{b_\alpha^2}{4}$ for all $\alpha$. Let $(w_\alpha)_\alpha$ be a sequence of positive functions such that
\begin{equation}\label{EqtBlowUp6}
\Delta_{g_\alpha}^2w_\alpha + b_\alpha\Delta_{g_\alpha}w_\alpha + c_\alpha w_\alpha = w_\alpha^{2^\sharp-1}
\end{equation}
in $B_0(2\delta)$ for all $\alpha$. Assume $\Vert w_\alpha\Vert_{L^\infty\left(B_0(\delta)\right)} \to +\infty$ as $\alpha \to +\infty$ 
and that there exists $C > 0$ such that $\Vert w_\alpha\Vert_{L^\infty\left(B_0(\delta)\backslash B_0(\delta/2)\right)} \le C$ 
for all $\alpha$. Then 
\begin{equation}\label{EqtBlowUp7}
\int_{B_0(\delta)}w_\alpha^{2^\sharp}dx \ge \left(1+o(1)\right)K_n^{-n/4}
\hskip.1cm ,
\end{equation}
where $o(1) \to 0$ as $\alpha \to +\infty$ and $K_n$ is the sharp constant in the Euclidean Sobolev inequality $\Vert u\Vert_{2^\sharp}^2 \le K_n\Vert\Delta u\Vert_2^2$,  
$u \in \dot H^2$. Here  $\dot H^2$ is the completion of $C^\infty_0(\mathbb{R}^n)$ with respect to $u \to \Vert\Delta u\Vert_2$. 
\end{lem}

\begin{proof}[Proof of Lemma \ref{Lem1}] Let $\nu_\alpha$ be given by 
$$\nu_\alpha^{2-\frac{n}{2}} = \max_{\overline{B_0(\delta)}}w_\alpha$$
and $x_\alpha$ be such that $w_\alpha(x_\alpha) = \nu_\alpha^{2-\frac{n}{2}}$. By assumption, 
$\nu_\alpha \to 0$ as $\alpha \to +\infty$ and $x_\alpha \in B_0(\delta/2)$ for $\alpha \gg 1$. Let $\tilde w_\alpha$ be given by 
$$\tilde w_\alpha(x) = \nu_\alpha^{\frac{n-4}{2}}w_\alpha\left(x_\alpha + \nu_\alpha x\right)\hskip.1cm .$$
For any $R > 1$, $\tilde w_\alpha$ is defined in $B_0(R)$ provided $\alpha \gg 1$ is sufficiently large. Let $d_{1,\alpha}$ 
and $d_{2,\alpha}$ be given by
\begin{equation}\label{SplittCsts}
d_{1,\alpha} = \frac{b_\alpha}{2} + \sqrt{\frac{b_\alpha^2}{4}-c_\alpha}
\hskip.2cm\hbox{and}\hskip.2cm d_{2,\alpha} = \frac{b_\alpha}{2} - \sqrt{\frac{b_\alpha^2}{4}-c_\alpha}
\hskip.1cm .
\end{equation}
Then, as is easily checked, for any $R > 1$,
\begin{equation}\label{EqtLem1Pf1}
\begin{split}
&\Delta^2_{\tilde g_\alpha}\tilde w_\alpha + b_\alpha\nu_\alpha^2\Delta_{\tilde g_\alpha}\tilde w_\alpha + c_\alpha\nu_\alpha^4\tilde w_\alpha\\
&= \left(\Delta_{\tilde g_\alpha} + d_{1,\alpha}\nu_\alpha^2\right)\circ\left(\Delta_{\tilde g_\alpha} + d_{2,\alpha}\nu_\alpha^2\right)\tilde w_\alpha\\
&= \tilde w_\alpha^{2^\sharp-1}
\end{split}
\end{equation}
in $B_0(R)$ for all $\alpha \gg 1$ sufficiently large, where $\tilde g_\alpha(x) = g_\alpha(x_\alpha + \nu_\alpha x)$. 
Since $0 \le \tilde w_\alpha \le 1$, it follows from \eqref{EqtLem1Pf1} and classical elliptic 
theory, as developped in Gilbarg-Trudinger \cite{GilTru}, that the $\tilde w_\alpha$'s are bounded in $C^{4,\theta}_{loc}\left(\mathbb{R}^n\right)$, 
$\theta \in (0,1)$. In particular, there exists $\tilde w$ such that, up to a subsequence, $\tilde w_\alpha \to \tilde w$ in $C^4_{loc}\left(\mathbb{R}^n\right)$ as 
$\alpha \to +\infty$. By rescaling invariance rules there also holds that $\tilde w \in H^2$. Passing to the limit as $\alpha \to +\infty$ in \eqref{EqtLem1Pf1}, we get 
that
$$\Delta^2\tilde w = \tilde w^{2^\sharp-1}\hskip.1cm .$$
Since $\tilde w(0) = \max_{\mathbb{R}^n}\tilde w = 1$, it follows from Lin's classification \cite{Lin} that $\tilde w$ is the ground state in \eqref{GroundStateShape0} below. 
Noting that for any $R > 1$,
\begin{eqnarray*}
\int_{B_0(R)}\tilde w_\alpha^{2^\sharp}dv_{\tilde g_\alpha} 
&=& \int_{B_{x_\alpha}(R\nu_\alpha)}w_\alpha^{2^\sharp}dv_{g_\alpha}\\
&\le& \left(1+o(1)\right)\int_{B_0(\delta)}w_\alpha^{2^\sharp}dx\hskip.1cm ,
\end{eqnarray*}
we get from the $L^\infty_{loc}$-convergence $\tilde w_\alpha \to \tilde w$  that 
\begin{equation}\label{EqtLem1Pf12}
\liminf_{\alpha \to +\infty}\int_{B_0(\delta)}w_\alpha^{2^\sharp}dx \ge \int_{B_0(R)}\tilde w^{2^\sharp}dx\hskip.1cm .
\end{equation}
Noting that $\tilde w$ is an extremal function for the sharp inequality $\Vert u\Vert_{2^\sharp}^2 \le K_n\Vert\Delta u\Vert_2^2$, 
it is easily seen that $\int_{\mathbb{R}^n}\tilde w^{2^\sharp}dx = K_n^{-n/4}$. Letting $R \to +\infty$ in  \eqref{EqtLem1Pf12}, this ends the proof of the lemma.
\end{proof}

From now on we let $B: \mathbb{R}^n \to \mathbb{R}$ be the ground state 
given by
\begin{equation}\label{GroundStateShape0}
B(x) = \left(1 + \frac{\vert x\vert^2}{\sqrt{\lambda_n}}\right)^{-\frac{n-4}{2}}
\hskip.1cm ,
\end{equation}
where $\lambda_n$ is as in \eqref{EqtBlowUp2}. Given $x \in M$, $\mu > 0$ and $u: M \to \mathbb{R}$, we also 
define the function $R_x^\mu u$ by
\begin{equation}\label{GroundStateShape}
R_x^\mu u(y) = \mu^{\frac{n-4}{2}}u\left(\exp_x(\mu y)\right)
\hskip.1cm ,
\end{equation}
where $y \in \mathbb{R}^n$ is such that $\vert y\vert < \frac{i_g}{\mu}$, and $i_g > 0$ is the injectivity radius of $(M,g)$. For $(u_\alpha)_\alpha$ 
as above, and $i \in \left\{1,\dots,k\right\}$, we define $\mathcal{S}_{i,r}, \mathcal{S}_{i,t} \subset \mathbb{R}^n$ by
\begin{equation}\label{DefSsets}
\begin{split}
&\mathcal{S}_{i,t} = \left\{\lim_{\alpha \to +\infty}\frac{1}{\mu_{i,\alpha}}\exp_{x_{i,\alpha}}^{-1}(x_{j,\alpha}), j = 1,\dots,k\right\}\\
&\mathcal{S}_{i,r} = \left\{\lim_{\alpha \to +\infty}\frac{1}{\mu_{i,\alpha}}\exp_{x_{i,\alpha}}^{-1}(x_{j,\alpha}), j \in I_i\right\}
\hskip.1cm ,
\end{split}
\end{equation}
where $I_i$ is the subset of $\left\{1,\dots,k\right\}$ consisting in the $j$'s which are such that 
$d_g(x_{i,\alpha},x_{j,\alpha}) = O\left(\mu_{i,\alpha}\right)$ and $\mu_{j,\alpha} = o\left(\mu_{i,\alpha}\right)$. 
The second lemma we prove establishes local limits for the $u_\alpha$'s. 

\begin{lem}\label{Lem2} Let $(M,g)$ be a smooth compact Riemannian manifold of dimension $n \ge 5$, and $b, c > 0$ be positive real 
numbers such that $c - \frac{b^2}{4} < 0$. Let also $(b_\alpha)_\alpha$ and $(c_\alpha)_\alpha$ be converging sequences 
of real numbers with limits $b$ and $c$ as $\alpha \to +\infty$, and $(u_\alpha)_\alpha$ be a bounded sequence in $H^2$ of 
positive nontrivial solutions of \eqref{PertCritEqt} satisfying \eqref{BlowUpAssumpt}. Then, up to a subsequence,
\begin{equation}\label{EqtLem2}
R_{x_{i,\alpha}}^{\mu_{i,\alpha}}u_\alpha \to B
\end{equation}
in $C^4_{loc}\left(\mathbb{R}^n\backslash\mathcal{S}_{i,r}\right)$ as $\alpha \to +\infty$ for all $i$, where $\mathcal{S}_{i,r}$ is 
as in \eqref{DefSsets}.
\end{lem}

\begin{proof}[Proof of Lemma \ref{Lem2}] First we claim that for any $i$ and any $K \subset\subset \mathbb{R}^n\backslash \mathcal{S}_{i,t}$, 
there exists $C_K > 0$ such that 
\begin{equation}\label{L1.2Eqt1}
\left\vert R_{x_{i,\alpha}}^{\mu_{i,\alpha}}u_\alpha\right\vert \le C_K
\end{equation}
in $K$, for all $\alpha \gg 1$ sufficiently large. Fix $i$ and $K$. For any $x \in K$, and any $j$, 
\begin{eqnarray*} d_g\left(\exp_{x_{i,\alpha}}(\mu_{i,\alpha}x),x_{j,\alpha}\right)
&\ge& C\left\vert\mu_{i,\alpha}x - \exp_{x_{i,\alpha}}^{-1}(x_{j,\alpha})\right\vert\\
&\ge& C \mu_{i,\alpha}\left\vert x - \frac{1}{\mu_{i,\alpha}}\exp_{x_{i,\alpha}}^{-1}(x_{j,\alpha})\right\vert\\
&\ge& C\mu_{i,\alpha}
\end{eqnarray*}
by the definition of $\mathcal{S}_{i,t}$ in \eqref{DefSsets}. Then \eqref{L1.2Eqt1} follows from \eqref{EqtBlowUp4}. Also, by direct 
computations, using the structure equation \eqref{EqtBlowUp3}, 
there holds that for any $i\not= j$, $R_{x_{i,\alpha}}^{\mu_{i,\alpha}}B_\alpha^j \to 0$ in $L^\infty_{loc}\left(\mathbb{R}^n\backslash\mathcal{S}_{i,r}\right)$ 
as $\alpha \to +\infty$, where the $B_\alpha^i$'s are as in \eqref{EqtBlowUp2}. Noting that $R_{x_{i,\alpha}}^{\mu_{i,\alpha}}B_\alpha^i(x) = \eta(\mu_{i,\alpha}x)B(x)$, where 
$\eta$ is as in \eqref{EqtBlowUp2} and $B$ is as in 
\eqref{GroundStateShape0}, we get that for any $i$,
\begin{equation}\label{L1.2Eqt2}
R_{x_{i,\alpha}}^{\mu_{i,\alpha}}\sum_{j=1}^kB_\alpha^j \to B
\end{equation}
in $L^\infty_{loc}\left(\mathbb{R}^n\backslash\mathcal{S}_{i,r}\right)$ as $\alpha \to +\infty$. Now we prove \eqref{EqtLem2}. 
By \eqref{EqtBlowUp1}, 
\begin{equation}\label{L1.2Eqt3}
R_{x_{i,\alpha}}^{\mu_{i,\alpha}}u_\alpha - R_{x_{i,\alpha}}^{\mu_{i,\alpha}}\sum_{j=1}^kB_\alpha^j \to 0
\end{equation}
in $L^{2^\sharp}_{loc}(\mathbb{R}^n)$ as $\alpha \to +\infty$. Moreover, in any compact subset of $\mathbb{R}^n$,
\begin{equation}\label{L1.2Eqt4}
\Delta_{g_\alpha}^2\tilde u_\alpha + b_\alpha\mu_{i,\alpha}^2\Delta_{g_\alpha}\tilde u_\alpha + c_\alpha\mu_{i,\alpha}^4\tilde u_\alpha 
= \tilde u_\alpha^{2^\sharp-1}
\end{equation}
for $\alpha \gg 1$ sufficiently large, where $\tilde u_\alpha = R_{x_{i,\alpha}}^{\mu_{i,\alpha}}u_\alpha$ and $g_\alpha(x) = \bigl(\exp_{x_{i,\alpha}}^\star g\bigr)(\mu_{i,\alpha}x)$. 
Since $\mu_{i,\alpha} \to 0$, we have that $g_\alpha \to \xi$ in $C^4_{loc}(\mathbb{R}^n)$ as $\alpha \to +\infty$. By \eqref{L1.2Eqt1}, the sequence $(\tilde u_\alpha)_\alpha$ is bounded 
in $L^\infty_{loc}(\mathbb{R}^n\backslash\mathcal{S}_{i,t})$. By \eqref{L1.2Eqt2} and \eqref{L1.2Eqt3}, there also holds that
\begin{equation}\label{L1.2Eqt5}
\lim_{\delta \to 0}\limsup_{\alpha\to+\infty} \int_{B_x(\delta)}\tilde u_\alpha^{2^\sharp}dx = 0
\end{equation}
for all $x \in \mathbb{R}^n\backslash\mathcal{S}_{i,r}$. By Lemma \ref{Lem1} we then get that the sequence $(\tilde u_\alpha)_\alpha$ is actually bounded 
in $L^\infty_{loc}(\mathbb{R}^n\backslash\mathcal{S}_{i,r})$. By \eqref{L1.2Eqt4} and elliptic theory it follows that $(\tilde u_\alpha)_\alpha$ is bounded 
in $C^{4,\theta}_{loc}(\mathbb{R}^n\backslash\mathcal{S}_{i,r})$, $\theta \in (0,1)$. This ends the proof of Lemma \ref{Lem2}.
\end{proof}

The following lemma establishes pointwise estimates for the $u_\alpha$'s. The estimates in Lemma \ref{Lem3} are a trace extension of the estimates 
\eqref{EqtBlowUp4}. In particular, as is easily checked, \eqref{PtseEstLem3} below implies \eqref{EqtBlowUp4}.

\begin{lem}\label{Lem3} Let $(M,g)$ be a smooth compact Riemannian manifold of dimension $n \ge 5$, and $b, c > 0$ be positive real 
numbers such that $c - \frac{b^2}{4} < 0$. Let also $(b_\alpha)_\alpha$ and $(c_\alpha)_\alpha$ be converging sequences 
of real numbers with limits $b$ and $c$ as $\alpha \to +\infty$, and $(u_\alpha)_\alpha$ be a bounded sequence in $H^2$ of 
positive nontrivial solutions of \eqref{PertCritEqt} satisfying \eqref{BlowUpAssumpt}. Then, up to a subsequence,
\begin{equation}\label{PtseEstLem3}
r_\alpha^4\left\vert u_\alpha - u_\infty - \sum_{i=1}^kB_\alpha^i\right\vert^{2^\sharp-2} \to 0
\end{equation}
in $L^\infty(M)$ as $\alpha \to +\infty$, where $r_\alpha$ is as in \eqref{DefrAlpha}, and the $B_\alpha^i$'s are as in \eqref{EqtBlowUp2}.
\end{lem}

\begin{proof}[Proof of Lemma \ref{Lem3}] Let $D_\alpha: M \to \mathbb{R}$ be given by
\begin{equation}\label{DefDAlpha}
D_\alpha(x) = \min_{i=1,\dots,k}\left(d_g(x_{i,\alpha},x) + \mu_{i,\alpha}\right)\hskip.1cm .
\end{equation}
We prove sligthly more than \eqref{PtseEstLem3}, namely that
\begin{equation}\label{PtseEstLem3Bis}
D_\alpha^4\left\vert u_\alpha - u_\infty - \sum_{i=1}^kB_\alpha^i\right\vert \to 0
\end{equation}
in $L^\infty(M)$ as $\alpha \to +\infty$, where $D_\alpha$ is as in \eqref{DefDAlpha}. Let $x_\alpha \in M$ be such that
\begin{equation}\label{DefxAlpha}
\begin{split}
&D_\alpha^4(x_\alpha)\left\vert u_\alpha(x_\alpha) - u_\infty(x_\alpha) - \sum_{i=1}^kB_\alpha^i(x_\alpha)\right\vert^{2^\sharp-2}\\
&= \max_{x\in M}D_\alpha^4(x)\left\vert u_\alpha(x) - u_\infty(x) - \sum_{i=1}^kB_\alpha^i(x)\right\vert^{2^\sharp-2} \hskip.1cm .
\end{split}
\end{equation}
First we claim that
\begin{equation}\label{PfLem1.3Eqt1}
\begin{split}
&\liminf_{\alpha\to +\infty}D_\alpha^4(x_\alpha)\left\vert u_\alpha(x_\alpha) - u_\infty(x_\alpha) - \sum_{i=1}^kB_\alpha^i(x_\alpha)\right\vert^{2^\sharp-2} > 0\\
&\Longrightarrow \lim_{\alpha\to +\infty}D_\alpha(x_\alpha)^4B_\alpha^i(x_\alpha)^{2^\sharp-2} = 0
\end{split}
\end{equation}
for all $i$. In order to prove \eqref{PfLem1.3Eqt1} we proceed by contradiction and assume that there exists 
$\varepsilon_0 > 0$ and $i$ such that $D_\alpha(x_\alpha)^4B_\alpha^i(x_\alpha)^{2^\sharp-2} \ge \varepsilon_0$, and thus 
such that
\begin{equation}\label{PfLem1.3Eqt2}
\eta\left(d_g(x_{i,\alpha},x_\alpha)\right)\left(\frac{D_\alpha(x_\alpha)\mu_{i,\alpha}}{\mu_{i,\alpha}^2 + \frac{d_g(x_{i,\alpha},x_\alpha)^2}{\sqrt{\lambda_n}}}\right)^4 \ge \varepsilon_0
\hskip.1cm .
\end{equation}
By \eqref{PfLem1.3Eqt2} we get that $d_g(x_{i,\alpha},x_\alpha) \to 0$ as $\alpha \to +\infty$, that there exists $\lambda \ge 0$ such that, up to a subsequence,
\begin{equation}\label{PfLem1.3Eqt3}
\frac{d_g(x_{i,\alpha},x_\alpha)}{\mu_{i,\alpha}} \to \lambda
\end{equation}
as $\alpha \to +\infty$, and that
\begin{equation}\label{PfLem1.3Eqt4}
\frac{\mu_{j,\alpha}}{\mu_{i,\alpha}} + \frac{d_g(x_{j,\alpha},x_\alpha)}{\mu_{i,\alpha}} \ge \varepsilon_0^{1/4}
\end{equation}
for all $\alpha$ and $j$. Let $y_\alpha$ be given by
$$y_\alpha = \frac{1}{\mu_{i,\alpha}}\exp_{x_{i,\alpha}}^{-1}(x_\alpha)\hskip.1cm .$$
By \eqref{PfLem1.3Eqt4}, $d(y_\alpha,\mathcal{S}_{i,r}) \ge \varepsilon$ for all $\alpha$, where $\varepsilon > 0$ is independent of $\alpha$, while by 
\eqref{PfLem1.3Eqt3} there holds that $\vert y_\alpha\vert \le C$ for all $\alpha$, where $C > 0$ is independent of $\alpha$. We have that 
$D_\alpha(x_\alpha) \le \mu_{i,\alpha}$ by \eqref{PfLem1.3Eqt3}. By Lemma \ref{Lem2} we then get that 
\begin{equation}\label{PfLem1.3Eqt5}
D_\alpha^4(x_\alpha)\left\vert u_\alpha(x_\alpha) - u_\infty(x_\alpha) - B_\alpha^i(x_\alpha)\right\vert^{2^\sharp-2} \to 0
\end{equation}
as $\alpha \to +\infty$. As in the proof of Lemma \ref{Lem2}, using the structure equation \eqref{EqtBlowUp3}, 
there also holds that $D_\alpha(x_\alpha)^4B_\alpha^j(x_\alpha)^{2^\sharp-2} \to 0$ 
as $\alpha \to +\infty$ for all $j \not= i$. Coming back to \eqref{PfLem1.3Eqt5}, the contradiction follows with the assumption in 
\eqref{PfLem1.3Eqt1}. This proves \eqref{PfLem1.3Eqt1}. 
Now we prove \eqref{PtseEstLem3Bis}. Here again we proceed by contradiction and assume 
that there exists $\varepsilon_0 > 0$ such that
\begin{equation}\label{NewContrAsspt}
D_\alpha^4(x_\alpha)\left\vert u_\alpha(x_\alpha) - u_\infty(x_\alpha) - \sum_{i=1}^kB_\alpha^i(x_\alpha)\right\vert^{2^\sharp-2} \ge \varepsilon_0
\hskip.1cm ,
\end{equation}
where $x_\alpha$ is as in \eqref{DefxAlpha}. We claim that
\begin{equation}\label{Pf2Lem1.3Eqt1}
u_\alpha(x_\alpha) \to +\infty
\end{equation}
as $\alpha \to +\infty$. By \eqref{PfLem1.3Eqt1} and \eqref{NewContrAsspt}, we get \eqref{Pf2Lem1.3Eqt1} if we prove that 
$D_\alpha(x_\alpha) \to 0$ as $\alpha \to +\infty$. Suppose on the contrary that, up to a subsequence, 
$D_\alpha(x_\alpha) \to \delta$ as $\alpha \to +\infty$ for some $\delta > 0$. By \eqref{PfLem1.3Eqt1} and \eqref{NewContrAsspt} 
there holds that
\begin{equation}\label{Pf2Lem1.3Eqt2}
\left\vert u_\alpha(x) - u_\infty(x)\right\vert^{2^\sharp-2} \le C\left\vert u_\alpha(x_\alpha)-u_\infty(x_\alpha)\right\vert^{2^\sharp-2} + o(1) 
\end{equation}
for all $x \in B_{x_\alpha}(\delta/2)$, and all $\alpha \gg 1$. Assuming that \eqref{Pf2Lem1.3Eqt1} is false, it follows from \eqref{Pf2Lem1.3Eqt2} that the 
$u_\alpha$'s are uniformly bounded in a neighborhood of the $x_\alpha$'s. By elliptic theory we then get that $u_\alpha \to u_\infty$ in 
$L^\infty(\Omega)$, where $\Omega$ is a neighborhood of the limit of the $x_\alpha$'s. Hence $u_\alpha(x_\alpha)-u_\infty(x_\alpha) \to 0$ and 
we get a contradiction with \eqref{PfLem1.3Eqt1} and \eqref{NewContrAsspt}. This proves \eqref{Pf2Lem1.3Eqt1}. 
Now we let $\mu_\alpha$ be given by $\mu_\alpha^{2-\frac{n}{2}} = u_\alpha(x_\alpha)$
and define $\tilde u_\alpha$ by $\tilde u_\alpha = R_{x_\alpha}^{\mu_\alpha}u_\alpha$. Then
\begin{equation}\label{Pf2Lem1.3Eqt3}
\Delta_{g_\alpha}^2\tilde u_\alpha + b_\alpha\mu_\alpha^2\Delta_{g_\alpha}\tilde u_\alpha 
+ c_\alpha\mu_\alpha^4\tilde u_\alpha = \tilde u_\alpha^{2^\sharp-1}\hskip.1cm ,
\end{equation}
where $g_\alpha(x) = \bigl(\exp_{x_\alpha}^\star g\bigr)(\mu_\alpha x)$. By \eqref{Pf2Lem1.3Eqt1}, $\mu_\alpha \to 0$ as $\alpha \to +\infty$. 
It follows that $g_\alpha \to \xi$ in $C^4_{loc}(\mathbb{R}^n)$ as $\alpha \to +\infty$. Also there holds that $\tilde u_\alpha(0) = 1$ and that 
the $\tilde u_\alpha$'s are bounded in $\dot H^2$. Up to a subsequence, $\tilde u_\alpha \rightharpoonup \tilde u_\infty$ in 
$H^2_{loc}$ and 
\begin{equation}\label{Pf2Lem1.3Eqt4}
\Delta^2\tilde u_\infty = \tilde u_\infty^{2^\sharp-1}
\hskip.1cm ,
\end{equation}
where $\Delta = -\hbox{div}_\xi\nabla$ is the Euclidean Laplacian. Let $\tilde{\mathcal{S}}$ be given by
$$\tilde{\mathcal{S}} = \left\{\lim_{\alpha \to +\infty}\frac{1}{\mu_\alpha}\exp_{x_\alpha}^{-1}(x_{i,\alpha})~,~i \in I\right\}
\hskip.1cm ,$$
where $I$ consists of the indices $i$ which are such that $d_g(x_{i,\alpha},x_\alpha) = O(\mu_\alpha)$ 
and $\mu_{i,\alpha} = o(\mu_\alpha)$. In what follows we let $K \subset\subset \mathbb{R}^n\backslash\tilde{\mathcal{S}}$ be compact, 
$x \in K$. By \eqref{DefxAlpha}, \eqref{PfLem1.3Eqt1} and \eqref{NewContrAsspt}, we have that
\begin{equation}\label{Pf2Lem1.3Eqt5}
\begin{split}
&\left\vert \tilde u_\alpha(x) - \mu_\alpha^{\frac{n-4}{2}}u_\infty(y_\alpha) - 
\mu_\alpha^{\frac{n-4}{2}}\sum_{i=1}^kB_\alpha^i(y_\alpha)\right\vert^{2^\sharp-2}\\
&\le \left(\frac{D_\alpha(x_\alpha)}{D_\alpha(y_\alpha)}\right)^4\bigl(1+o(1)\bigr) + o(1)
\hskip.1cm ,
\end{split}
\end{equation}
where $y_\alpha = \exp_{x_\alpha}(\mu_\alpha x)$. It can be checked that
\begin{equation}\label{Pf2Lem1.3Eqt6}
\mu_\alpha^{\frac{n-4}{2}}B_\alpha^i(y_\alpha) \to 0
\end{equation}
for all $i$, as $\alpha \to +\infty$. By \eqref{Pf2Lem1.3Eqt1}, \eqref{Pf2Lem1.3Eqt5} and \eqref{Pf2Lem1.3Eqt6} we then get that 
\begin{equation}\label{Pf2Lem1.3Eqt7}
\tilde u_\alpha(x)^{2^\sharp-2} \le \left(\frac{D_\alpha(x_\alpha)}{D_\alpha(y_\alpha)}\right)^4\bigl(1+o(1)\bigr) + o(1)
\hskip.1cm .
\end{equation}
Since $x \in K$, and $K\subset\subset\mathbb{R}^n\backslash\tilde{\mathcal{S}}$, there holds that 
$D_\alpha(x_\alpha) = O\left(D_\alpha(y_\alpha)\right)$. Hence, by \eqref{Pf2Lem1.3Eqt7}, for any 
$K\subset\subset\mathbb{R}^n\backslash\tilde{\mathcal{S}}$, there exists $C > 0$ such that $\vert\tilde u_\alpha\vert \le C$ in $K$. 
By elliptic theory and  \eqref{Pf2Lem1.3Eqt3} we then get that
\begin{equation}\label{Pf2Lem1.3Eqt8}
\tilde u_\alpha \to \tilde u_\infty
\end{equation}
in $C^4_{loc}(\mathbb{R}^n\backslash\tilde{\mathcal{S}})$ as $\alpha \to +\infty$. It holds that $0 \not\in \tilde{\mathcal{S}}$ since, 
if not the case, we can write that $D_\alpha(x_\alpha) = o(\mu_\alpha)$ and we get a contradiction with \eqref{NewContrAsspt}. 
As a consequence, since $\tilde u_\alpha(0) = 1$, it follows that $\tilde u_\infty(0) = 1$ and thus that $\tilde u_\infty \not\equiv 0$. 
By \eqref{Pf2Lem1.3Eqt4} and Lin's classification \cite{Lin}  we then get that
\begin{equation}\label{Pf2Lem1.3Eqt9}
\tilde u_\infty(x) = \left(\frac{\lambda}{\lambda^2+\frac{\vert x-x_0\vert^2}{\sqrt{\lambda_n}}}\right)^{\frac{n-4}{2}}
\end{equation}
for some $\lambda > 0$ and $x_0 \in \mathbb{R}^n$ such that $\sqrt{\lambda_n}\lambda(\lambda-1) + \vert x_0\vert^2 = 0$. 
Let $K \subset\subset \mathbb{R}^n\backslash\tilde{\mathcal{S}}$, $K\not= \emptyset$. By \eqref{EqtBlowUp1} and 
\eqref{Pf2Lem1.3Eqt6} there holds that $\tilde u_\alpha \to 0$ in 
$L^{2^\sharp}(K)$. By \eqref{Pf2Lem1.3Eqt8} we should get that $\int_K\tilde u_\infty^{2^\sharp}dx = 0$, a contradiction with 
\eqref{Pf2Lem1.3Eqt9}. This ends the proof of Lemma \ref{Lem3}. 
\end{proof}

As already mentioned, it follows from \eqref{PtseEstLem3} that there exists $C > 0$ such that
\begin{equation}\label{WeakEstOrder0}
r_\alpha^{(n-4)/2}u_\alpha \le C\hskip.1cm .
\end{equation}
Derivative companions to this estimate are as follows.

\begin{lem}\label{Lem3bis} Let $(M,g)$ be a smooth compact Riemannian manifold of dimension $n \ge 5$, and $b, c > 0$ be positive real 
numbers such that $c - \frac{b^2}{4} < 0$. Let also $(b_\alpha)_\alpha$ and $(c_\alpha)_\alpha$ be converging sequences 
of real numbers with limits $b$ and $c$ as $\alpha \to +\infty$, and $(u_\alpha)_\alpha$ be a bounded sequence in $H^2$ of 
positive nontrivial solutions of \eqref{PertCritEqt} satisfying \eqref{BlowUpAssumpt}. There exists $C > 0$ such that, up to a subsequence,
\begin{equation}\label{PtseEstLem3bis}
r_\alpha^{\frac{n-4}{2}+k}\vert\nabla^ku_\alpha\vert \le C
\end{equation}
in $M$ for all $\alpha$, and all $k = 1,2,3$, where $r_\alpha$ is as in \eqref{DefrAlpha}.
\end{lem}

\begin{proof}[Proof of Lemma \ref{Lem3bis}] Lemma \ref{Lem3bis} follows from Green's representation formula 
and \eqref{PtseEstLem3}. Let $G_\alpha$ be the Green's function of $\Delta_g^2 + b_\alpha\Delta_g + c_\alpha$. By Green's 
representation formula,
\begin{equation}\label{EqtAddedEstDer1}
u_\alpha(x) = \int_MG_\alpha(x,y)u_\alpha(y)^{2^\sharp-1}dv_g(y)
\end{equation}
for all $x \in M$. Then, by \eqref{EqtAddedEstDer1},
\begin{equation}\label{EqtAddedEstDer2}
\nabla^ku_\alpha(x) = \int_M\nabla^k_xG_\alpha(x,y)u_\alpha(y)^{2^\sharp-1}dv_g(y)
\end{equation}
for all $x \in M$. There holds, see Grunau and Robert \cite{GruRob}, that
\begin{equation}\label{Green'sEstimate1st}
\left\vert\nabla_y^kG_\alpha(x,y)\right\vert \le Cd_g(x,y)^{4-n-k}
\end{equation}
for all $\alpha$, all $x, y \in M$ with $x \not= y$, and all $k \in \left\{0,1,2,3\right\}$. By \eqref{WeakEstOrder0}, \eqref{EqtAddedEstDer2}, 
\eqref{Green'sEstimate1st} and Giraud's lemma, we then get that 
there exists $C > 0$ such that
\begin{eqnarray*} \left\vert\nabla^ku_\alpha(x)\right\vert 
&\le& \int_M\vert\nabla_x^kG_\alpha(x,y)\vert u_\alpha(y)^{2^\sharp-1}dv_g(y)\\
&\le& C\sum_{i=1}^N\int_Md_g(x,y)^{4-n-k}d_g(x_{i,\alpha},y)^{-(n+4)/2}dv_g(y)\\
&\le&Cr_\alpha(x)^{-\frac{n-4}{2}-k}\hskip.1cm .
\end{eqnarray*}
This proves \eqref{PtseEstLem3bis}.
\end{proof}

Estimates such as in \eqref{EqtBlowUp4} and Lemma \ref{Lem3bis} are scale invariant estimates. When 
transposed to the Euclidean space, in the simple case of a single isolated 
blow-up point, we would get, for instance when $k = 0$ and $k = 2$, that 
$\vert x\vert^{(n-4)/2}\vert u(x)\vert \le C$ and $\vert x\vert^{n/2}\vert\Delta u(x)\vert \le C$. 
These two estimates are  
invariant with respect to the scaling $\lambda^{(n-4)/2}u(\lambda x)$ which leaves 
invariant both the equation $\Delta^2u = u^{2^\sharp-1}$ and the $\dot H^2$-norm.  
Scale invariant estimates, together with the Sobolev description \eqref{EqtBlowUp1}, provide valuable 
informations on the $u_\alpha$'s. However they are still not strong enough to conclude 
to the theorem. Sharper estimates such as the ones in Proposition \ref{SharpPtEst} are required. From now on we define the $v_\alpha$'s by 
\begin{equation}\label{DefvAlpha}
v_\alpha = \Delta_gu_\alpha + \frac{b_\alpha}{2}u_\alpha
\end{equation}
for all $\alpha$. By \eqref{PertCritEqt} there holds that
\begin{equation}\label{EqtvAlpha}
\Delta_gv_\alpha + \frac{b_\alpha}{2}v_\alpha = \tilde c_\alpha u_\alpha + u_\alpha^{2^\sharp-1}\hskip.1cm ,
\end{equation}
where $\tilde c_\alpha = \frac{b_\alpha^2}{4}-c_\alpha$. 
In particular, when $c_\alpha \le \frac{b_\alpha^2}{4}$, we get by the maximum principle that 
either $v_\alpha > 0$ in $M$ or $v_\alpha \equiv 0$. When we also assume that $b_\alpha > 0$, this implies 
that $v_\alpha > 0$ when $u_\alpha$ is nontrivial. 
The following lemma is a key point toward the proof of Proposition \ref{SharpPtEst}. 

\begin{lem}\label{Lem4} Let $(M,g)$ be a smooth compact Riemannian manifold of dimension $n \ge 5$. 
Let also $(b_\alpha)_\alpha$ and $(c_\alpha)_\alpha$ be converging sequences 
of positive real numbers satisfying that $c_\alpha - \frac{b_\alpha^2}{4} \le 0$ 
for all $\alpha$, and $(u_\alpha)_\alpha$ be a bounded sequence in $H^2$ of 
positive nontrivial solutions of \eqref{PertCritEqt} satisfying \eqref{BlowUpAssumpt}. Let 
$u_\infty$ be such that, up to a subsequence, $u_\alpha \to u_\infty$ a.e. in $M$. There exists $C_1 > 0$ such that 
$v_\alpha \ge C_1u_\alpha^{2^\sharp/2}$ in $M$ for all $\alpha$, where the 
$v_\alpha$'s are as in \eqref{DefvAlpha}. Assuming that either $u_\infty \not\equiv 0$ or $c - \frac{b^2}{4} < 0$, where 
$b$ and $c$ are the limits of the sequences $(b_\alpha)_\alpha$ and $(c_\alpha)_\alpha$, 
there also exists  $C_2 > 0$ such that $v_\alpha \ge C_2u_\alpha$ in $M$ for all $\alpha$.
\end{lem}

\begin{proof}[Proof of Lemma \ref{Lem4}] We use twice the basic remark that if $\Omega$ is an open subset of $M$, $u, v$ 
are $C^2$-positive functions in $\Omega$, 
and $x_0 \in \Omega$ is a point where $\frac{v}{u}$ 
achieves its supremum in $\Omega$, then
\begin{equation}\label{AddedLem}
\frac{\Delta_gv(x_0)}{v(x_0)} \ge \frac{\Delta_gu(x_0)}{u(x_0)}\hskip.1cm .
\end{equation}
Indeed,
$\nabla\left(\frac{v}{u}\right) = \frac{u\nabla v - v\nabla u}{u^2}$
so that $u(x_0)\nabla v(x_0) = v(x_0)\nabla u(x_0)$. Then,
$$\Delta_g\left(\frac{v}{u}\right)(x_0) = \frac{u(x_0)\Delta_gv(x_0) - v(x_0)\Delta_gu(x_0)}{u^2(x_0)}$$
and we get \eqref{AddedLem} by writing that $\Delta_g\left(\frac{v}{u}\right)(x_0) \ge 0$. In what follows we 
let $x_\alpha \in M$ be such that
\begin{equation}\label{Lem4Eqt1}
\frac{u_\alpha(x_\alpha)^{2^\sharp/2}}{v_\alpha(x_\alpha)} = 
\max_{x \in M}\frac{u_\alpha^{2^\sharp/2}}{v_\alpha}\hskip.1cm .
\end{equation}
Then, by \eqref{AddedLem}, 
\begin{equation}\label{Lem4Eqt2}
\frac{\Delta_gu_\alpha(x_\alpha)^{2^\sharp/2}}{u_\alpha(x_\alpha)^{2^\sharp/2}} \ge \frac{\Delta_gv_\alpha(x_\alpha)}{v_\alpha(x_\alpha)}
\hskip.1cm .
\end{equation}
We compute
\begin{equation}\label{Lem4Eqt3}
\Delta_gu_\alpha^{\frac{2^\sharp}{2}} = \frac{2^\sharp}{2}u_\alpha^{\frac{2^\sharp}{2}-1}\Delta_gu_\alpha 
- \frac{2^\sharp}{2}\left(\frac{2^\sharp}{2}-1\right)u_\alpha^{\frac{2^\sharp}{2}-2}\vert\nabla u_\alpha\vert^2\hskip.1cm .
\end{equation}
It follows from \eqref{Lem4Eqt2} and \eqref{Lem4Eqt3} that 
\begin{equation}\label{Lem4Eqt4}
\frac{2^\sharp}{2}\frac{\Delta_gu_\alpha(x_\alpha)}{u_\alpha(x_\alpha)} \ge \frac{\Delta_gv_\alpha(x_\alpha)}{v_\alpha(x_\alpha)}
\hskip.1cm .
\end{equation}
By \eqref{EqtvAlpha} and \eqref{Lem4Eqt4} we then get that
\begin{eqnarray*} 
\frac{2^\sharp}{2}\frac{v_\alpha(x_\alpha)}{u_\alpha(x_\alpha)} 
&=& \frac{2^\sharp}{2}\frac{\Delta_gu_\alpha(x_\alpha)}{u_\alpha(x_\alpha)} + \frac{2^\sharp}{2}\frac{b_\alpha}{2}\\
&\ge& \frac{\Delta_gv_\alpha(x_\alpha)}{v_\alpha(x_\alpha)} + \frac{2^\sharp}{2}\frac{b_\alpha}{2}\\
&=&  \frac{\Delta_gv_\alpha(x_\alpha)+\frac{b_\alpha}{2}v_\alpha(x_\alpha)}{v_\alpha(x_\alpha)} + \left(\frac{2^\sharp}{2}-1\right)\frac{b_\alpha}{2}\\
&=& \frac{u_\alpha(x_\alpha)^{2^\sharp-1}}{v_\alpha(x_\alpha)} + \tilde c_\alpha\frac{u_\alpha(x_\alpha)}{v_\alpha(x_\alpha)} + \left(\frac{2^\sharp}{2}-1\right)\frac{b_\alpha}{2}
\hskip.1cm ,
\end{eqnarray*}
where $\tilde c_\alpha = \frac{b_\alpha^2}{4}-c_\alpha$. By assumption, $b_\alpha > 0$ and $\tilde c_\alpha \ge 0$. In particular, we get that 
$$v_\alpha(x_\alpha) \ge \sqrt{\frac{2}{2^\sharp}}u_\alpha(x_\alpha)^{2^\sharp/2}$$
for all $\alpha$. The estimate $v_\alpha \ge Cu_\alpha^{2^\sharp/2}$ follows from the definition \eqref{Lem4Eqt1} of $x_\alpha$. This proves the first part of 
Lemma \ref{Lem4}. In order to get the second part we let $x_\alpha$ be such that
\begin{equation}\label{Lem4Eqt5}
\frac{u_\alpha(x_\alpha)}{v_\alpha(x_\alpha)} = \max_{x\in M} \frac{u_\alpha(x)}{v_\alpha(x)}\hskip.1cm .
\end{equation}
By \eqref{AddedLem} and \eqref{Lem4Eqt5}, 
$$\frac{\Delta_gu_\alpha(x_\alpha)}{u_\alpha(x_\alpha)} \ge \frac{\Delta_gv_\alpha(x_\alpha)}{v_\alpha(x_\alpha)}$$
and we get with \eqref{EqtvAlpha} that
\begin{eqnarray*}
\frac{v_\alpha(x_\alpha)-\frac{b_\alpha}{2}u_\alpha(x_\alpha)}{u_\alpha(x_\alpha)}
&\ge& \frac{\Delta_gv_\alpha(x_\alpha) + \frac{b_\alpha}{2}v_\alpha(x_\alpha)}{v_\alpha(x_\alpha)} - \frac{b_\alpha}{2}\\
&\ge& \frac{\tilde c_\alpha u_\alpha(x_\alpha) + u_\alpha(x_\alpha)^{2^\sharp-1}}{v_\alpha(x_\alpha)} - \frac{b_\alpha}{2}
\hskip.1cm .
\end{eqnarray*}
As a consequence, 
$$\frac{v_\alpha(x_\alpha)}{u_\alpha(x_\alpha)} \ge \tilde c_\alpha \frac{u_\alpha(x_\alpha)}{v_\alpha(x_\alpha)} 
+ u_\alpha(x_\alpha)^{2^\sharp-2}\frac{u_\alpha(x_\alpha)}{v_\alpha(x_\alpha)}$$
and we get that
\begin{equation}\label{Lem4Eqt6}
\frac{v_\alpha(x_\alpha)^2}{u_\alpha(x_\alpha)^2} 
\ge \tilde c_\alpha + u_\alpha(x_\alpha)^{2^\sharp-2}\hskip.1cm .
\end{equation}
Assuming that $c - \frac{b^2}{4} < 0$ there exists $\delta > 0$ such  that $c_\alpha \ge \delta$ for all $\alpha$. 
Similarly, let us assume that $u_\infty \not\equiv 0$. If $G_\alpha$ stands for the 
Green's function of $\Delta_g^2 + b_\alpha\Delta_g + c_\alpha$, then 
$$u_\alpha(x_\alpha)
= \int_MG_\alpha(x_\alpha,\cdot)u_\alpha^{2^\sharp-1}dv_g\\
\ge \int_{M\backslash\bigcup_iB_{x_i}(\delta^\prime)}G_\alpha(x_\alpha,\cdot)u_\alpha^{2^\sharp-1}dv_g\hskip.1cm .$$
By Lemma \ref{Lem3}, $u_\alpha \to u_\infty$ uniformly in compact subsets of $M\backslash\bigcup_{i=1}^k\{x_i\}$. 
Letting $\alpha \to +\infty$, and then $\delta^\prime \to 0$, it follows that there exists 
$\delta > 0$ such that $u_\alpha(x_\alpha) \ge \delta$ for all $\alpha$. In particular, in both cases $c - \frac{b^2}{4} < 0$ 
and $u_\infty \not\equiv 0$, we get with 
\eqref{Lem4Eqt6} that $v_\alpha(x_\alpha) \ge C u_\alpha(x_\alpha)$ 
for some $C > 0$ independent of $\alpha$. By the definition of $x_\alpha$ in \eqref{Lem4Eqt5} 
it follows that $v_\alpha \ge C u_\alpha$ in $M$ for all $\alpha$. This ends the proof of the lemma. 
\end{proof}

At that point, given $\delta > 0$, we define $\eta_\alpha(\delta)$ by
\begin{equation}\label{DefEtaAlpha}
\eta_\alpha(\delta) = \max_{M\backslash\bigcup_iB_{x_{i,\alpha}}(\delta)}v_\alpha
\hskip.1cm ,
\end{equation}
where $v_\alpha$ is as in \eqref{DefvAlpha}. Then we prove the following first set of pointwise $\varepsilon$-sharp estimates 
on the $u_\alpha$'s and $v_\alpha$'s.

\begin{lem}\label{Lem5} Let $(M,g)$ be a smooth compact Riemannian manifold of dimension $n \ge 5$, and $b, c > 0$ be positive real 
numbers such that $c - \frac{b^2}{4} < 0$. Let also $(b_\alpha)_\alpha$ and $(c_\alpha)_\alpha$ be converging sequences 
of real numbers with limits $b$ and $c$ as $\alpha \to \infty$, and $(u_\alpha)_\alpha$ be a bounded sequence in $H^2$ of 
positive nontrivial solutions of \eqref{PertCritEqt} satisfying \eqref{BlowUpAssumpt}. Let $0 < \varepsilon < \varepsilon_0$, where $\varepsilon_0 > 0$ 
is sufficiently small. There exist $R_\varepsilon > 0$, $\delta_\varepsilon > 0$, 
and $C_\varepsilon > 0$ such that
\begin{equation}\label{EpsilEstLem5}
\begin{split}
&u_\alpha \le C_\varepsilon\left(\mu_\alpha^{\frac{n-4}{2}-(n-2)\varepsilon}r_\alpha^{4-n+(n-2)\varepsilon} 
+ \eta_\alpha(\delta_\varepsilon)\right)~,\\
&v_\alpha \le C_\varepsilon\left(\mu_\alpha^{\frac{n-4}{2}-(n-2)\varepsilon}r_\alpha^{(2-n)(1-\varepsilon)} 
+ \eta_\alpha(\delta_\varepsilon)r_\alpha^{(2-n)\varepsilon}\right)
\end{split}
\end{equation}
in $M\backslash\bigcup_{i=1}^kB_{x_{i,\alpha}}(R_\varepsilon\mu_{i,\alpha})$ for all $\alpha$, where 
$\mu_{i,\alpha}$ is as in \eqref{EqtBlowUp2}, $r_\alpha$ is as 
in \eqref{DefrAlpha}, $\mu_\alpha$ is as in \eqref{DefMuAlpha},  $v_\alpha$ is as in \eqref{DefvAlpha}, and $\eta_\alpha$ is as 
in \eqref{DefEtaAlpha}.
\end{lem}

\begin{proof}[Proof of Lemma \ref{Lem5}] The first estimate we prove is the one on the $v_\alpha$'s from which we deduce then the estimate 
on the $u_\alpha$'s by using the Green's representation of $u_\alpha$ in  terms of $v_\alpha$. We establish the estimate on the $v_\alpha$'s for 
$0 < \varepsilon < \frac{1}{2}$, and the estimate on the $u_\alpha$'s for $0 < \varepsilon < \frac{1}{n-2}\min(2,n-4)$.

\medskip (1) {\it Proof of the estimate on the $v_\alpha$'s in \eqref{EpsilEstLem5}}. We fix $0 < \varepsilon < \frac{1}{2}$. Let $G^\prime_1$ be the Green's function of $\Delta_g+1$ and 
let $\psi_{\alpha,\varepsilon}$ be given by
$$\psi_{\alpha,\varepsilon}(x) = \mu_\alpha^{\frac{n-4}{2}-(n-2)\varepsilon}\sum_iG^\prime_1(x_{i,\alpha},x)^{1-\varepsilon} + 
\eta_\alpha(\delta_\varepsilon)\sum_iG^\prime_1(x_{i,\alpha},x)^\varepsilon\hskip.1cm .$$
Given $R > 0$ we let $\Omega_{\alpha,R} = \bigcup_iB_{x_{i,\alpha}}(R\mu_{i,\alpha})$, and let $x_\alpha \in M\backslash\Omega_{\alpha,R}$ be such that
$$\max_{M\backslash\Omega_{\alpha,R}}\frac{v_\alpha}{\psi_{\alpha,\varepsilon}} = \frac{v_\alpha(x_\alpha)}{\psi_{\alpha,\varepsilon}(x_\alpha)}\hskip.1cm .$$
First we claim that for $\delta_\varepsilon \ll 1$ and $R_\varepsilon \gg 1$ suitably chosen,
\begin{equation}\label{Eqt1L5}
x_\alpha \in \partial\left(M\backslash\Omega_{\alpha,R}\right)\hskip.2cm\hbox{or}\hskip.2cm r_\alpha(x_\alpha) \ge \delta_\varepsilon\hskip.1cm .
\end{equation}
We prove \eqref{Eqt1L5} by contradiction. We assume $x_\alpha \not\in \partial\left(M\backslash\Omega_{\alpha,R}\right)$ and $r_\alpha(x_\alpha) < \delta_\varepsilon$. 
We have that
\begin{equation}\label{AddedFREqt1}
\frac{\Delta_gv_\alpha(x_\alpha)}{v_\alpha(x_\alpha)} \ge \frac{\Delta_g\psi_{\alpha,\varepsilon}(x_\alpha)}{\psi_{\alpha,\varepsilon}(x_\alpha)}
\end{equation}
and by direct computations, using standard properties of the Green's function $G^\prime_1$ such as its control by the distance to the pole, there also holds that 
since $0 < \varepsilon < \frac{1}{2}$, there exist $C_0(\varepsilon), C_1(\varepsilon) > 0$ such that
\begin{equation}\label{AddedFREqt2}
r_\alpha(x_\alpha)^2\frac{\Delta_g\psi_{\alpha,\varepsilon}(x_\alpha)}{\psi_{\alpha,\varepsilon}(x_\alpha)} \ge C_0(\varepsilon) 
- C_1(\varepsilon)r_\alpha(x_\alpha)^2\hskip.1cm .
\end{equation}
By \eqref{EqtvAlpha} and Lemma \ref{Lem4},
\begin{equation}\label{AddedFREqt3}
\begin{split}
r_\alpha(x_\alpha)^2\frac{\Delta_gv_\alpha(x_\alpha)}{v_\alpha(x_\alpha)} 
&\le r_\alpha(x_\alpha)^2 \frac{u_\alpha(x_\alpha)^{2^\sharp-1}}{v_\alpha(x_\alpha)} + r_\alpha(x_\alpha)^2\tilde c_\alpha\frac{u_\alpha(x_\alpha)}{v_\alpha(x_\alpha)}\\
&\le Cr_\alpha(x_\alpha)^2u_\alpha(x_\alpha)^{\frac{2^\sharp-2}{2}} + Cr_\alpha(x_\alpha)^2
\end{split}
\end{equation}
for all $\alpha$, where $\tilde c_\alpha = \frac{b_\alpha^2}{4}-c_\alpha$ and $C > 0$ does not depend on $\alpha$. By Lemma \ref{Lem3}, 
\begin{equation}\label{AddedFREqt4}
\begin{split}
\left(r_\alpha(x_\alpha)^2u_\alpha(x_\alpha)^{\frac{2^\sharp-2}{2}}\right)^2 
&\le Cr_\alpha(x_\alpha)^4 + C\sum_ir_\alpha(x_\alpha)^4B_\alpha^i(x_\alpha)^{2^\sharp-2} + o(1)\\
&\le C\delta_\varepsilon^4 + \varepsilon_R\hskip.1cm ,
\end{split}
\end{equation}
where $\varepsilon_R \to 0$ as $R \to +\infty$. By \eqref{AddedFREqt1}--\eqref{AddedFREqt4} we get a contradiction by choosing $\delta_\varepsilon \ll 1$ and $R \gg 1$ 
sufficiently small. This proves \eqref{Eqt1L5}. Now that we have \eqref{Eqt1L5}, up to increasing $R$, we claim that thanks to 
Lemma \ref{Lem2},
\begin{equation}\label{Eqt2L5}
\max_{M\backslash\Omega_{\alpha,R}}\frac{v_\alpha}{\psi_{\alpha,\varepsilon}} \le C_\varepsilon\hskip.1cm .
\end{equation}
Indeed, if $r_\alpha(x_\alpha) \ge \delta_\varepsilon$, then
$$\frac{v_\alpha(x_\alpha)}{\psi_{\alpha,\varepsilon}(x_\alpha)} \le \frac{v_\alpha(x_\alpha)}{\eta_\alpha(\delta_\varepsilon)}
\left(\sum_iG^\prime_1(x_{i,\alpha},x_\alpha)^\varepsilon\right)^{-1} 
\le C\hskip.1cm ,$$
while if $x_\alpha \in \partial\left(M\backslash\Omega_{\alpha,R}\right)$, we get that
\begin{eqnarray*} v_\alpha(x_\alpha)
&=& v_\alpha\left(\exp_{x_{i,\alpha}}(\mu_{i,\alpha}z_\alpha)\right)\\
&=& \mu_{i,\alpha}^{-\frac{n}{2}}\left(\Delta_{g_\alpha}(R_{x_{i,\alpha}}^{\mu_{i,\alpha}}u_\alpha) + \frac{b_\alpha\mu_{i,\alpha}^2}{2}(R_{x_{i,\alpha}}^{\mu_{i,\alpha}}u_\alpha)\right)(z_\alpha)
\end{eqnarray*}
for some $i$, where $z_\alpha$ is such that $x_\alpha = \exp_{x_{i,\alpha}}(\mu_{i,\alpha}z_\alpha)$, 
and $g_\alpha(x) = \bigl(\exp_{x_{i,\alpha}}^\star g\bigr)(\mu_{i,\alpha}x)$. Then, by 
Lemma \ref{Lem2}, and standard properties of $G^\prime_1$,
$$\frac{v_\alpha(x_\alpha)}{\psi_{\alpha,\varepsilon}(x_\alpha)}
\le C \frac{\mu_{i,\alpha}^{-\frac{n}{2}}\mu_{i,\alpha}^{(n-2)(1-\varepsilon)}}{\mu_\alpha^{\frac{n-4}{2}-(n-2)\varepsilon}}
\le C\left(\frac{\mu_{i,\alpha}}{\mu_\alpha}\right)^{\frac{n-4}{2} - (n-2)\varepsilon}\le C$$
up to choosing $R \gg 1$ such that $\partial B_0(R)\bigcap \mathcal{S}_{i,r} = \emptyset$, where $\mathcal{S}_{i,r}$ is as in \eqref{DefSsets}.
In particular, we get that \eqref{Eqt2L5} holds true. Noting that $G^\prime_1(x_{i,\alpha},x) \le Cr_\alpha(x)^{-(n-2)}$, this ends the proof of the 
estimate on the $v_\alpha$'s in \eqref{EpsilEstLem5}.

\medskip (2) {\it Proof of the estimate on the $u_\alpha$'s in \eqref{EpsilEstLem5}}. We fix $0 < \varepsilon < \frac{1}{n-2}\min(2,n-4)$. 
Let $G^\prime_2$ be the Green's function of $\Delta_g + \frac{b}{4}$. Let $(x_\alpha)_\alpha$ be an arbitrary sequence of points such that 
$x_\alpha \in M\backslash\Omega_{\alpha,R}$ for all $\alpha$, where $R > 0$ is to be chosen later on. There holds that
\begin{equation}\label{RepForEqt1}
\begin{split}
u_\alpha(x_\alpha) 
&= \int_MG^\prime_2(x_\alpha,x)\left(\bigl(\Delta_g+\frac{b}{4}\bigr)u_\alpha\right)(x)dv_g(x)\\
&\le \int_MG^\prime_2(x_\alpha,x)v_\alpha(x)dv_g(x)\\
&\le \int_{M\backslash\Omega_{\alpha,R_\varepsilon}}G^\prime_2(x_\alpha,x)v_\alpha(x)dv_g(x)\\
&\hskip.4cm + \sum_i\int_{B_{x_{i,\alpha}(R_\varepsilon\mu_{i,\alpha})}}G^\prime_2(x_\alpha,x)v_\alpha(x)dv_g(x)
\end{split}
\end{equation}
for all $\alpha$, where $R_\varepsilon$ is the radius obtained when proving 
the estimate on the $v_\alpha$'s in \eqref{EpsilEstLem5}. We have that $G^\prime_2(x_\alpha,x) \le Cd_g(x_\alpha,x)^{2-n}$. Hence, by Giraud's lemma,
\begin{equation}\label{EstVarepsEqt1}
\begin{split}
&\int_{M\backslash\Omega_{\alpha,R_\varepsilon}}G^\prime_2(x_\alpha,x)r_\alpha(x)^{(2-n)(1-\varepsilon)}dv_g(x)\\
&\le C\sum_i\int_Md_g(x_\alpha,x)^{2-n}d_g(x_{i,\alpha},x)^{2+(n-2)\varepsilon-n}dv_g(x)\\
&\le C\sum_id_g(x_{i,\alpha},x_\alpha)^{4-n+(n-2)\varepsilon}
\end{split}
\end{equation}
since $\varepsilon < \frac{n-4}{n-2}$. Now we fix $R > 0$ such that $R \ge 2R_\varepsilon$. Then $d_g(x_\alpha,x) \ge \frac{1}{2}d_g(x_{i,\alpha},x)$ 
for all $x \in B_{x_{i,\alpha}}(R_\varepsilon\mu_{i,\alpha})$, and we get that
\begin{equation}\label{EstVarepsEqt2}
\begin{split}
&\int_{B_{x_{i,\alpha}}(R_\varepsilon)}G^\prime_2(x_\alpha,x)v_\alpha(x)dv_g(x)\\
&\le Cd_g(x_{i,\alpha}x_\alpha)^{2-n}\int_{B_{x_{i,\alpha}}(R_\varepsilon)}v_\alpha dv_g\\
&\le Cd_g(x_{i,\alpha}x_\alpha)^{2-n}\hbox{Vol}_g\bigl(B_{x_{i,\alpha}}(R_\varepsilon\mu_{i,\alpha})\bigr)^{1/2}\Vert v_\alpha\Vert_{L^2(M)}\\
&\le Cd_g(x_{i,\alpha}x_\alpha)^{2-n}\mu_{i,\alpha}^{n/2}
\end{split}
\end{equation}
since the $v_\alpha$'s are bounded in $L^2$. At last, still by Giraud's lemma, 
\begin{equation}\label{EstVarepsEqt3}
\begin{split}
&\int_{M\backslash\Omega_{\alpha,R_\varepsilon}}G^\prime_2(x_\alpha,x)r_\alpha(x)^{(2-n)\varepsilon}dv_g(x)\\
&\le C\sum_i\int_Md_g(x_\alpha,x)^{2-n}d_g(x_{i,\alpha},x)^{(2-n)\varepsilon}dv_g(x)\\
&\le C
\end{split}
\end{equation}
since $0 < \varepsilon < \frac{2}{n-2}$. 
Combining \eqref{EstVarepsEqt1}--\eqref{EstVarepsEqt3} with \eqref{RepForEqt1}, thanks to the estimate on the $v_\alpha$'s in 
\eqref{EpsilEstLem5}, we get that 
\begin{equation}\label{EstVarepsEqt4}
\begin{split}
u_\alpha(x_\alpha) 
&\le C\mu_\alpha^{\frac{n-4}{2}-(n-2)\varepsilon}\sum_id_g(x_{i,\alpha},x_\alpha)^{4-n+(n-2)\varepsilon}\\
&+ C\sum_i\mu_{i,\alpha}^{n/2}d_g(x_{i,\alpha}x_\alpha)^{2-n} + C \eta_\alpha(\delta_\varepsilon)
\end{split}
\end{equation}
There holds that
\begin{eqnarray*}
&&\mu_\alpha^{\frac{n-4}{2}-(n-2)\varepsilon}d_g(x_{i,\alpha},x_\alpha)^{4-n+(n-2)\varepsilon} 
+ \mu_{i,\alpha}^{n/2}d_g(x_{i,\alpha}x_\alpha)^{2-n}\\
&&= \mu_\alpha^{\frac{n-4}{2}-(n-2)\varepsilon}d_g(x_{i,\alpha},x_\alpha)^{4-n+(n-2)\varepsilon}\\
&&\hskip.4cm\times\left(1+d_g(x_{i,\alpha},x_\alpha)^{-2-(n-2)\varepsilon}\mu_{i,\alpha}^{n/2}\mu_\alpha^{-\frac{n-4}{2}+(n-2)\varepsilon}\right)
\end{eqnarray*}
and
$$d_g(x_{i,\alpha},x_\alpha)^{-2-(n-2)\varepsilon}\mu_{i,\alpha}^{n/2}\mu_\alpha^{-\frac{n-4}{2}+(n-2)\varepsilon} 
\le C \left(\frac{\mu_{i,\alpha}}{\mu_\alpha}\right)^{\frac{n-4}{2}-(n-2)\varepsilon} \le C$$
since $d_g(x_{i,\alpha},x_\alpha) \ge R\mu_{i,\alpha}$. Coming back to \eqref{EstVarepsEqt4} we get that
the estimate on the $u_\alpha$'s in \eqref{EpsilEstLem5} holds true. This ends the proof of the lemma.
\end{proof}

Thanks to the estimates in Lemma \ref{Lem5} we can now prove Proposition \ref{SharpPtEst}.

\begin{proof}[Proof of Proposition \ref{SharpPtEst}] Consider the estimates: there exist $C > 0$, $R > 0$ and $\delta > 0$ such that, up to a subsequence,
\begin{equation}\label{EqtMnPrpToProve}
\left\vert\nabla^ju_\alpha\right\vert \le C\left(\mu_\alpha^{\frac{n-4}{2}}r_\alpha^{4-n-j} + \eta_\alpha(\delta)^{2^\sharp-1}\right)
\end{equation}
in $M\backslash\Omega_{\alpha,R}$ for all $j = 0,1,2,3$, and all $\alpha$, where $\Omega_{\alpha,R} = \bigcup_iB_{x_{i,\alpha}}(R\mu_\alpha)$. 
We prove Proposition \ref{SharpPtEst} by proving first these estimates, then by proving that we can replace $\eta_\alpha(\delta)^{2^\sharp-1}$ by $\Vert u_\infty\Vert_{L^\infty}$ 
in \eqref{EqtMnPrpToProve}, and at last by proving that the estimates hold in the whole of $M$.

\medskip (1) {\it Proof of \eqref{EqtMnPrpToProve}}. Let $G_\alpha$ be the Green's function of the fourth order 
operator $\Delta_g^2 + b_\alpha\Delta_g + c_\alpha$. By Lemma \ref{Lem5}, given $0 < \varepsilon \ll 1$, there exist 
$R_\varepsilon, C_\varepsilon, \delta_\varepsilon > 0$ such that
\begin{equation}\label{TotEstEqt1}
u_\alpha \le C_\varepsilon\left(\mu_\alpha^{\frac{n-4}{2}-(n-2)\varepsilon}r_\alpha^{4-n+(n-2)\varepsilon} 
+ \eta_\alpha(\delta_\varepsilon)\right)
\end{equation}
in $M\backslash\bigcup_{i=1}^kB_{x_{i,\alpha}}(R_\varepsilon\mu_{i,\alpha})$. Since $\mu_{i,\alpha} \le \mu_\alpha$ by the definition of $\mu_\alpha$, there 
holds that $M\backslash\Omega_{\alpha,R} \subset M\backslash\bigcup_{i=1}^kB_{x_{i,\alpha}}(R\mu_{i,\alpha})$. Let $(x_\alpha)_\alpha$ be a 
sequence in $M\backslash\Omega_{\alpha,R}$. We aim at proving that there exists $C, \delta > 0$ such that, up to a subsequence, 
\begin{equation}\label{TotEstEqt1ToProve}
\left\vert\nabla^ju_\alpha(x_\alpha)\right\vert \le C\left(\mu_\alpha^{\frac{n-4}{2}}r_\alpha(x_\alpha)^{4-n-j} + \eta_\alpha(\delta)^{2^\sharp-1}\right)
\end{equation}
for all $\alpha$. Let $R = 2R_\varepsilon+1$ and $\delta = \delta_\varepsilon$. We have that 
\begin{equation}\label{GreenFctRel}
u_\alpha(x) = \int_MG_\alpha(x,y)u_\alpha(y)^{2^\sharp-1}dv_g(y)
\end{equation}
for all $\alpha$ and $x \in M$. By combining \eqref{Green'sEstimate1st} and \eqref{GreenFctRel} we then get that 
\begin{equation}\label{FirstSplittEqt1}
\begin{split}
\left\vert\nabla^ju_\alpha(x_\alpha)\right\vert 
&\le \int_M\left\vert\nabla^j_xG_\alpha(x_\alpha,x)\right\vert u_\alpha(x)^{2^\sharp-1}dv_g(x)\\
&\le C\int_Md_g(x_\alpha,x)^{4-n-j}u_\alpha(x)^{2^\sharp-1}dv_g(x)\\
&\le \int_{M\backslash\Omega_{\alpha,R_\varepsilon}}d_g(x_\alpha,x)^{4-n-j}u_\alpha(x)^{2^\sharp-1}dv_g(x)\\ 
&\hskip.4cm + \int_{\Omega_{\alpha,R_\varepsilon}}d_g(x_\alpha,x)^{4-n-j}u_\alpha(x)^{2^\sharp-1}dv_g(x)
\end{split}
\end{equation}
Let $k_n = (2^\sharp-1)(n-2)$. By \eqref{TotEstEqt1}, 
\begin{equation}\label{FirstSplittEqt2}
\begin{split}
&\int_{M\backslash\Omega_{\alpha,R_\varepsilon}}d_g(x_\alpha,x)^{4-n-j}u_\alpha(x)^{2^\sharp-1}dv_g(x)\\
&\le C\mu_\alpha^{\frac{n+4}{2}-k_n\varepsilon}\sum_iA_\alpha^i
+ C\eta_\alpha(\delta_\varepsilon)^{2^\sharp-1}\hskip.1cm ,
\end{split}
\end{equation}
where
$$A_\alpha^i = \int_{M\backslash B_{x_{i,\alpha}}(R_\varepsilon\mu_\alpha)}d_g(x_\alpha,x)^{4-n-j}d_g(x_{i,\alpha},x)^{-(n+4)+k_n\varepsilon}dv_g(x)
\hskip.1cm .$$
Let $K_{i,\alpha} = \bigl\{x~\hbox{s.t.}~d_g(x_\alpha,x) \le \frac{1}{2}d_g(x_{i,\alpha},x_\alpha)\bigr\}$. Then
\begin{equation}\label{FirstSplittEqt3}
\begin{split}
A_\alpha^i
&\le \int_{K_{i,\alpha}\backslash B_{x_{i,\alpha}}(R_\varepsilon\mu_\alpha)}d_g(x_\alpha,x)^{4-n-j}d_g(x_{i,\alpha},x)^{-(n+4)+k_n\varepsilon}dv_g(x)\\
&\hskip.4cm + \int_{K_{i,\alpha}^c\backslash B_{x_{i,\alpha}}(R_\varepsilon\mu_\alpha)}d_g(x_\alpha,x)^{4-n-j}d_g(x_{i,\alpha},x)^{-(n+4)+k_n\varepsilon}dv_g(x)\\
&= A_{1,\alpha}^i + A_{2,\alpha}^i
\hskip.1cm .
\end{split}
\end{equation}
By the definition of $K_{i,\alpha}$, there holds that $d_g(x_\alpha,x) \le d_g(x_{i,\alpha},x)$ in $K_{i,\alpha}$. Hence, choosing $\varepsilon \ll 1$ sufficiently small 
such that $4-k_n\varepsilon > 0$, we can write that
\begin{equation*}
\begin{split}
&\mu_\alpha^{4-k_n\varepsilon}\int_{K_{i,\alpha}\backslash B_{x_{i,\alpha}}(R_\varepsilon\mu_\alpha)}d_g(x_\alpha,x)^{4-n-j}d_g(x_{i,\alpha},x)^{-(n+4)+k_n\varepsilon}dv_g(x)\\
&\le \int_{K_{i,\alpha}\backslash B_{x_{i,\alpha}}(R_\varepsilon\mu_\alpha)}d_g(x_\alpha,x)^{4-n-j-\theta}d_g(x_{i,\alpha},x)^{\theta-n}
\left(\frac{d_g(x_\alpha,x)}{d_g(x_{i,\alpha},x)}\right)^\theta dv_g(x)\hskip.1cm ,
\end{split}
\end{equation*}
where $0 < \theta \ll 1$ is chosen small, and by Giraud's lemma we get that
\begin{equation}\label{FirstSplittEqt4}
A_{1,\alpha}^i \le C\mu_\alpha^{-4+k_n\varepsilon}d_g(x_{i,\alpha},x_\alpha)^{4-n-j}
\hskip.1cm .
\end{equation}
In $K_{i,\alpha}^c$ there holds that $d_g(x_\alpha,x) \ge \frac{1}{2}d_g(x_{i,\alpha},x_\alpha)$, and we can directly write that
\begin{equation}\label{FirstSplittEqt5}
\begin{split}
A_{2,\alpha}^i
&\le Cd_g(x_{i,\alpha},x_\alpha)^{4-n-j}\int_{K_{i,\alpha}^c\backslash B_{x_{i,\alpha}}(R_\varepsilon\mu_\alpha)}d_g(x_{i,\alpha},x)^{-(n+4)+k_n\varepsilon}dv_g(x)\\
&\le C\mu_\alpha^{-4+k_n\varepsilon}d_g(x_{i,\alpha},x_\alpha)^{4-n-j}\hskip.1cm .
\end{split}
\end{equation}
At last, since $R \ge 2R_\varepsilon + 1$, we can write that $d_g(x_\alpha,x) \ge \frac{1}{2}d_g(x_{i,\alpha},x_\alpha)$ for all 
$\alpha$ and all $x \in B_{x_{i,\alpha}}(R_\varepsilon\mu_\alpha)$. Hence, by H\"older's inequality, for any $i$,
\begin{equation}\label{FirstSplittEqt6}
\begin{split}
&\int_{B_{x_{i,\alpha}}(R_\varepsilon\mu_\alpha)}d_g(x_\alpha,x)^{4-n-j}u_\alpha(x)^{2^\sharp-1}dv_g(x)\\
&\le C d_g(x_\alpha,x_{i,\alpha})^{4-n-j}\int_{B_{x_{i,\alpha}}(R_\varepsilon\mu_\alpha)}u_\alpha(x)^{2^\sharp-1}dv_g(x)\\
&\le C d_g(x_\alpha,x_{i,\alpha})^{4-n-j}\hbox{Vol}_g\left(B_{x_{i,\alpha}}(R_\varepsilon\mu_\alpha)\right)^{1-\frac{2^\sharp-1}{2^\sharp}}\Vert u_\alpha\Vert_{L^{2^\sharp}}^{2^\sharp-1}\\
&\le C \mu_\alpha^{\frac{n-4}{2}} r_\alpha(x_\alpha)^{4-n-j}\hskip.1cm .
\end{split}
\end{equation}
Combining \eqref{FirstSplittEqt1}--\eqref{FirstSplittEqt6} we 
then get that \eqref{TotEstEqt1ToProve} holds true. This proves \eqref{EqtMnPrpToProve}.

\medskip (2) {\it Proof that \eqref{EqtMnPrpToProve} holds true with $\Vert u_\infty\Vert_{L^\infty}$ instead of $\eta_\alpha(\delta)^{2^\sharp-1}$}. If $u_\infty \not\equiv 0$ there is nothing to do 
since, by Lemma \ref{Lem3bis}, $\eta_\alpha(\delta) \le C$.  
Now we prove that
\begin{equation}\label{ToBeProvedEta}
\eta_\alpha(\delta) \le C\mu_\alpha^{(n-4)/2}
\end{equation}
in case $u_\infty \equiv 0$. This is sufficient to conclude to the validity of 
\eqref{EqtMnPrpToProve} with $\Vert u_\infty\Vert_{L^\infty}$ instead of $\eta_\alpha(\delta)^{2^\sharp-1}$. We assume in what follows that $u_\infty \equiv 0$ and 
we define $\Omega_\alpha(\delta) = \bigcup_iB_{x_{i,\alpha}}(\delta)$. By \eqref{EqtMnPrpToProve},
\begin{equation}\label{EstEtaEqt1}
\max_{M\backslash\Omega_\alpha(\delta/2)}u_\alpha \le C\mu_\alpha^{\frac{n-4}{2}} + C \eta_\alpha(\delta)^{2^\sharp-1}
\hskip.1cm ,
\end{equation}
while, by \eqref{EqtvAlpha}, we can write that
\begin{equation}\label{EstEtaEqt2}
\max_{M\backslash\Omega_\alpha(\delta)}v_\alpha \le C\max_{M\backslash\Omega_\alpha(\delta/2)}u_\alpha + C\Vert v_\alpha\Vert_{L^1}
\hskip.1cm .
\end{equation}
Let $G^\prime_3$ be the Green's function of $\Delta_g + b$. There exists $\Lambda > 0$ such that 
$G^\prime_3 \ge \Lambda$ in $M$ and since $v_\alpha \le (\Delta_g + b)u_\alpha$ 
for $\alpha \gg 1$, we get with Green's representation formula that 
$$\Lambda\Vert v_\alpha\Vert_{L^1} \le \int_MG^\prime_3(x,y)v_\alpha(y)dv_g(y) \le u_\alpha(x)$$
for all $\alpha$ and all $x$. Therefore, thanks to \eqref{EstEtaEqt2},
\begin{equation}\label{EstEtaEqt3}
\eta_\alpha(\delta) \le C\max_{M\backslash\Omega_\alpha(\delta/2)}u_\alpha
\hskip.1cm .
\end{equation}
By Lemma \ref{Lem3}, there holds that
$\max_{M\backslash\Omega_\alpha(\delta/2)}u_\alpha\to 0$ as $\alpha \to +\infty$. 
In particular $\eta_\alpha(\delta) \to 0$ as $\alpha \to +\infty$, and since by \eqref{EstEtaEqt1} and 
\eqref{EstEtaEqt3},
$$\max_{M\backslash\Omega_\alpha(\delta/2)}u_\alpha \le C\mu_\alpha^{\frac{n-4}{2}} + C \eta_\alpha(\delta)^{2^\sharp-2}
\max_{M\backslash\Omega_\alpha(\delta/2)}u_\alpha\hskip.1cm ,$$
we get that 
\begin{equation}\label{EstEtaEqt4}
\max_{M\backslash\Omega_\alpha(\delta/2)}u_\alpha \le C\mu_\alpha^{\frac{n-4}{2}}\hskip.1cm .
\end{equation}
The existence of $C > 0$ such that \eqref{ToBeProvedEta} holds true follows from \eqref{EstEtaEqt3} and \eqref{EstEtaEqt4}. 

\medskip (3) {\it Proof that the estimates are global in $M$}. According to the preceding discussion, the estimates \eqref{EqtMnPrp} hold
in $M\backslash\Omega_{\alpha,R}$ for some $R > 0$. We are left with the proof that they also hold in $\Omega_{\alpha,R}$. By Lemmas 
\ref{Lem3} and \ref{Lem3bis},
$$r_\alpha^{\frac{n-4}{2}+j}\vert\nabla^ju_\alpha\vert \le C$$
in $M$ for all $j=0,1,2,3$. Noting that $r_\alpha^{-\frac{n-2}{4}-j} \le Cr_\alpha^{4-n-j}$ 
in $\Omega_{\alpha,R}$, this ends the proof 
of the proposition.
\end{proof}

\section{Proof of Theorem \ref{LimitProfileThm}}\label{ProofThm1}

We prove Theorem \ref{LimitProfileThm} by contradiction. We assume that $(M,g)$ is conformally flat of dimension $n \ge 5$. 
We let $(b_\alpha)_\alpha$ and $(c_\alpha)_\alpha$ be converging sequences 
of real numbers with limits $b$ and $c$ as $\alpha \to \infty$, $c - \frac{b^2}{4} < 0$, and $(u_\alpha)_\alpha$ be a bounded sequence in $H^2$ of 
positive nontrivial solutions of \eqref{PertCritEqt} satisfying \eqref{BlowUpAssumpt}. The Pohozaev identity for fourth order 
equations can be written as follows: for any smooth bounded domain $\Omega \subset \mathbb{R}^n$, and any 
$u \in C^4(\overline{\Omega})$,
\begin{equation}\label{PohTypeIdentity}
\begin{split}
&\int_\Omega\left(x^k\partial_ku\right)\Delta^2u dx 
+ \frac{n-4}{2} \int_\Omega u\Delta^2u dx\\
&= \frac{n-4}{2} \int_{\partial\Omega}\left(-u\frac{\partial\Delta u}{\partial\nu} + 
\frac{\partial u}{\partial\nu}\Delta u\right)d\sigma\\
&+ \int_{\partial\Omega}\left(\frac{1}{2}(x,\nu)(\Delta u)^2 - (x,\nabla u)\frac{\partial\Delta u}{\partial\nu} 
+ \frac{\partial(x,\nabla u)}{\partial\nu}\Delta u\right)d\sigma
\hskip.1cm ,
\end{split}
\end{equation}
where $\nu$ is the outward unit normal to $\partial\Omega$ and $d\sigma$ is the Euclidean 
volume element on $\partial\Omega$. A preliminary lemma we prove is concerned with the Pohozaev identity, 
applied to the $u_\alpha$'s, in balls of radii $\sqrt{\mu_\alpha}$, where $\mu_\alpha$ is as in \eqref{DefMuAlpha}. 
Without loss of generality, up to passing to a subsequence, we can suppose that $\mu_\alpha = \mu_{1,\alpha}$ 
for all $\alpha$. Then we let $x_\alpha = x_{1,\alpha}$ for all $\alpha$. We say $x_\alpha$ is the blow-up point associated with 
$\mu_\alpha$. The meaning of $\sqrt{\mu_\alpha}$ in this section is that it is precisely the distance up to which a bubble singularity like
in \eqref{EqtBlowUp2}, with $x_{i,\alpha} = x_\alpha$ and $\mu_{i,\alpha} = \mu_\alpha$, 
interact in the $L^\infty$-topology. Namely, for such a $B_\alpha$, 
$$\lim_{\alpha \to +\infty} \max_{M\backslash B_{x_\alpha}(R\sqrt{\mu_\alpha})}B_\alpha = \varepsilon_R\hskip.1cm ,$$
where $\varepsilon_R \to 0$ as $R \to +\infty$. In particular, $\max_{\partial B_{x_\alpha}(\delta_\alpha)}B_\alpha \to 0$ as 
$\alpha \to +\infty$ for any sequence $(\delta_\alpha)_\alpha$ of positive real numbers such that $\frac{\delta_\alpha}{\sqrt{\mu_\alpha}} \to +\infty$.

\begin{lem}\label{LemPrf1} Let $(M,g)$ be a smooth compact conformally flat 
Riemannian manifold of dimension $n \ge 5$, and $b, c > 0$ be positive real 
numbers such that $c - \frac{b^2}{4} < 0$. Let also $(b_\alpha)_\alpha$ and $(c_\alpha)_\alpha$ be converging sequences 
of real numbers with limits $b$ and $c$ as $\alpha \to \infty$, and $(u_\alpha)_\alpha$ be a bounded sequence in $H^2$ of 
positive nontrivial solutions of \eqref{PertCritEqt} satisfying \eqref{BlowUpAssumpt}. There exists $\delta > 0$ and $K(u_\infty) \ge 0$ 
such that $K(u_\infty) > 0$ if $u_\infty \not\equiv 0$, and
\begin{equation}\label{LemEstPohoz}
\begin{split}
&\int_{B_{x_\alpha}(\delta\sqrt{\mu_\alpha})}\bigl(A_g-b_\alpha g\bigr)(\nabla u_\alpha,\nabla u_\alpha)dv_g\\
&= o\left(\int_{B_{x_\alpha}(\delta\sqrt{\mu_\alpha})}\vert\nabla u_\alpha\vert^2dv_g\right) 
- \bigl(K(u_\infty) + o(1)\bigr)\mu_\alpha^{\frac{n-4}{2}}\hskip.1cm .
\end{split}
\end{equation}
for all $\alpha$, where $A_g$ is as in \eqref{DefAg}, $\mu_\alpha$ is as in \eqref{DefMuAlpha}, and $x_\alpha$ is the blow-up 
point associated with $\mu_\alpha$.
\end{lem}

\begin{proof}[Proof of Lemma \ref{LemPrf1}] Let $\overline{u}_\alpha$ be defined in bounded subsets of $\mathbb{R}^n$ 
by
\begin{equation}\label{DefUbarAlpha}
\overline{u}_\alpha(x) = u_\alpha\left(\exp_{x_\alpha}(\sqrt{\mu_\alpha}x)\right)\hskip.1cm .
\end{equation}
By Hebey, Robert and Wen \cite{HebRobWen}, there exist $\delta > 0$, $A > 0$, and a biharmonic 
function $\hat\varphi \in C^4\left(B_0(2\delta)\right)$ such that, up to a subsequence,
\begin{equation}\label{EqtLem7.1}
\overline{u}_\alpha(x) \to \frac{A}{\vert x\vert^{n-4}} + \hat\varphi(x)
\end{equation}
in $C^3_{loc}\left(B_0(2\delta)\backslash\{0\}\right)$ as $\alpha \to +\infty$, with the property that 
$\hat\varphi$ is nonnegative, and even positive in $B_0(2\delta)$ if $u_\infty \not\equiv 0$. Moreover, there also holds that 
for any $\alpha$,
\begin{equation}\label{EqtLem7.2}
\int_{B_{x_\alpha}(\delta\sqrt{\mu_\alpha})}u_\alpha^2dv_g 
= o(1)\int_{B_{x_\alpha}(\delta\sqrt{\mu_\alpha})}\vert\nabla u_\alpha\vert^2dv_g\hskip.1cm ,
\end{equation}
where $o(1) \to 0$ as $\alpha \to +\infty$. Now we let $x_\infty$ be the limit of the $x_\alpha$'s and let $\delta_0 > 0$ and $\hat g$ 
be such that $\hat g$ is flat in $B_{x_\infty}(4\delta_0)$. We write that $g = \varphi^{4/(n-4)}\hat g$ with 
$\varphi(x_\infty) = 1$, and let $\hat u_\alpha = u_\alpha\varphi$. Define
$$B_\alpha = \frac{4b_\alpha}{n-4} \varphi^{\frac{8-n}{n-4}}\hat g + \varphi^{\frac{12-n}{n-4}}A_g$$
and
\begin{eqnarray*}
&&h_\alpha = b_\alpha\varphi^{\frac{2}{n-4}}\Delta_{\hat g}\varphi^{\frac{2}{n-4}} 
- \frac{n-2}{4(n-1)} b_\alpha \varphi^{\frac{8}{n-4}} S_g 
+ c_\alpha \varphi^{\frac{8}{n-4}}\\
&&\hskip.4cm - \frac{n-4}{2}Q_g\varphi^{\frac{8}{n-4}} + \varphi^{\frac{n+4}{n-4}} 
\hbox{div}_g(A_gd\varphi^{-1})\hskip.1cm ,
\end{eqnarray*}
where $Q_g$ is the $Q$-curvature of $g$ and $A_g$ is as in \eqref{DefAg}. By conformal invariance of the geometric 
Paneitz operator in the left hand side of \eqref{ConfEqt}, there holds that
\begin{equation}\label{EuclEqt}
\begin{split}
&\Delta^2\hat u_\alpha + b_\alpha\varphi^{\frac{4}{n-4}}\Delta\hat u_\alpha 
- B_\alpha(\nabla\varphi,\nabla\hat u_\alpha) + h_\alpha\hat u_\alpha\\
&\hskip.4cm + \varphi^{\frac{n+4}{n-4}}\hbox{div}_g(\varphi^{-1}A_gd\hat u_\alpha) 
= \hat u_\alpha^{2^\sharp-1}
\end{split}
\end{equation}
in $B_{x_\infty}(4\delta)$, where $A_g$ is as in \eqref{DefAg}, $B_\alpha$, and $h_\alpha$ are as above, and 
$\Delta = \Delta_{\hat g}$ is the Euclidean Laplacian. As a remark, \eqref{EuclEqt} can be rewritten as
$$\Delta^2\hat u_\alpha +
\varphi^{\frac{4}{n-4}}\hbox{div}_\xi\left(\left(A_g-b_\alpha g\right)d\hat
u_\alpha\right) + \dots = \hat u^{2^\sharp-1}\hskip.1cm,$$
where the dots represent lower order terms. 
Now we let $\delta > 0$ be sufficiently small. We regard 
$\hat u_\alpha$ as a function in the Euclidean space and assimilate $x_\alpha$ to $0$ thanks to the exponential 
map $\exp_{x_\alpha}$ with respect to $g$. 
With an abusive use of notations, we still denote by $\varphi$ the function 
$\varphi\circ\exp_{x_\alpha}$, by $A_g$ the tensor field $(\exp_{x_\alpha})^\star A_g$, and 
by $\hat g$ the metric $(\exp_{x_\alpha})^\star\hat g$.  
Applying the Pohozaev identity 
(\ref{PohTypeIdentity}) to the $\hat u_\alpha$'s in $B_0(\delta\sqrt{\mu_\alpha})$ we get that
\begin{equation}\label{PohTypeIdentity2}
\begin{split}
&\int_{B_0(\delta\sqrt{\mu_\alpha})}\left(x^k\partial_k\hat u_\alpha\right)\Delta^2\hat u_\alpha dx 
+ \frac{n-4}{2} \int_{B_0(\delta\sqrt{\mu_\alpha})}\hat u_\alpha\Delta^2\hat u_\alpha dx\\
&= \frac{n-4}{2} \int_{\partial B_0(\delta\sqrt{\mu_\alpha})}
\left(-\hat u_\alpha\frac{\partial\Delta\hat u_\alpha}{\partial\nu} + 
\frac{\partial\hat u_\alpha}{\partial\nu}\Delta\hat u_\alpha\right)d\sigma\\
&+ \int_{\partial B_0(\delta\sqrt{\mu_\alpha})}\left(\frac{1}{2}(x,\nu)(\Delta\hat u_\alpha)^2 
- (x,\nabla\hat u_\alpha)\frac{\partial\Delta\hat u_\alpha}{\partial\nu} 
+ \frac{\partial(x,\nabla\hat u_\alpha)}{\partial\nu}\Delta\hat u_\alpha\right)d\sigma\hskip.1cm.
\end{split}
\end{equation}
Integrating by parts, using (\ref{EuclEqt}), we can also write that
\begin{equation}\label{PohTypeIdentityIntParts}
\begin{split}
&\int_{B_0(\delta\sqrt{\mu_\alpha})}\left(x^k\partial_k\hat u_\alpha\right)\Delta^2\hat u_\alpha dx 
+ \frac{n-4}{2} \int_{B_0(\delta\sqrt{\mu_\alpha})}\hat u_\alpha\Delta^2\hat u_\alpha dx\\
&= b_\alpha\int_{B_0(\delta\sqrt{\mu_\alpha})}\varphi^{\frac{4}{n-4}}\vert\nabla\hat u_\alpha\vert^2dx 
- \int_{B_0(\delta\sqrt{\mu_\alpha})}\varphi^{\frac{8}{n-4}}A_g(\nabla\hat u_\alpha,\nabla\hat u_\alpha)dx\\
&\hskip.4cm + o\left(\int_{B_0(\delta\sqrt{\mu_\alpha})}\vert\nabla\hat u_\alpha\vert^2dx\right) 
+ O\left(\int_{B_0(\delta\sqrt{\mu_\alpha})}\hat u_\alpha^2dx\right)\\
&\hskip.4cm + O\left(\int_{\partial B_0(\delta\sqrt{\mu_\alpha})}\hat u_\alpha^2(1 + \hat u_\alpha^{2^\sharp-2})dx\right)
+ O\left(\int_{\partial B_0(\delta\sqrt{\mu_\alpha})}\vert\nabla\hat u_\alpha\vert^2dx\right)\hskip.1cm ,
\end{split}
\end{equation}
where, in this equation, as already mentioned, we regard 
$\varphi$ and $A_g$ as defined in the Euclidean space. 
The proof of (\ref{PohTypeIdentityIntParts}) involves only 
straightforward computations. By \eqref{EqtLem7.1},
\begin{equation}\label{Sec6Eqt3}
\begin{split}
&\int_{\partial B_0(\delta\sqrt{\mu_\alpha})}\hat u_\alpha^2(1 + \hat u_\alpha^{2^\sharp-2})dx 
= o\left(\mu_\alpha^{\frac{n-4}{2}}\right)
\hskip.1cm ,\hskip.1cm\hbox{and}\\
&\int_{\partial B_0(\delta\sqrt{\mu_\alpha})}\vert\nabla\hat u_\alpha\vert^2dx = 
o\left(\mu_\alpha^{\frac{n-4}{2}}\right)
\end{split}
\end{equation}
while, by \eqref{EqtLem7.2},
\begin{equation}\label{Sec6Eqt4}
\int_{B_0(\delta\sqrt{\mu_\alpha})}\hat u_\alpha^2dx 
= o\left(\int_{B_0(\delta\sqrt{\mu_\alpha})}\vert\nabla\hat u_\alpha\vert^2dx\right)\hskip.1cm .
\end{equation}
Independently, we can also write with the change of variables $x = \sqrt{\mu_\alpha}y$ and 
\eqref{EqtLem7.1} that if $R_\alpha$ stands for the right hand side in 
(\ref{PohTypeIdentity2}), then
\begin{equation}\label{Sec6Eqt5}
\begin{split}
\mu_\alpha^{-\frac{n-4}{2}}R_\alpha
&\to \frac{n-4}{2} \int_{\partial B_0(\delta)}
\left(-\tilde u\frac{\partial\Delta\tilde u}{\partial\nu} + 
\frac{\partial\tilde u}{\partial\nu}\Delta\tilde u\right)d\sigma\\
&\hskip.2cm + \int_{\partial B_0(\delta)}\left(\frac{1}{2}(x,\nu)(\Delta\tilde u)^2 
- (x,\nabla\tilde u)\frac{\partial\Delta\tilde u}{\partial\nu} 
+ \frac{\partial(x,\nabla\tilde u)}{\partial\nu}\Delta\tilde u\right)d\sigma
\end{split}
\end{equation}
as $\alpha \to +\infty$, where
\begin{equation}\label{Sec6Eqt6}
\tilde u(x) = \frac{A}{\vert x\vert^{n-4}} + \hat\varphi(x)
\end{equation}
is given by \eqref{EqtLem7.1} (so that $\Delta^2\hat\varphi = 0$). Coming back to the Pohozaev identity 
\eqref{PohTypeIdentity}, taking $\Omega = B_0(\delta)\backslash B_0(r)$, 
and since $\Delta^2\tilde u = 0$ in $\Omega$, it comes that
\begin{equation}\label{Sec6Eqt7}
\begin{split}
&\frac{n-4}{2} \int_{\partial B_0(\delta)}
\left(-\tilde u\frac{\partial\Delta\tilde u}{\partial\nu} + 
\frac{\partial\tilde u}{\partial\nu}\Delta\tilde u\right)d\sigma\\
&\hskip.2cm + \int_{\partial B_0(\delta)}\left(\frac{1}{2}(x,\nu)(\Delta\tilde u)^2 
- (x,\nabla\tilde u)\frac{\partial\Delta\tilde u}{\partial\nu} 
+ \frac{\partial(x,\nabla\tilde u)}{\partial\nu}\Delta\tilde u\right)d\sigma\\
&= \frac{n-4}{2} \int_{\partial B_0(r)}
\left(-\tilde u\frac{\partial\Delta\tilde u}{\partial\nu} + 
\frac{\partial\tilde u}{\partial\nu}\Delta\tilde u\right)d\sigma\\
&\hskip.2cm + \int_{\partial B_0(r)}\left(\frac{1}{2}(x,\nu)(\Delta\tilde u)^2 
- (x,\nabla\tilde u)\frac{\partial\Delta\tilde u}{\partial\nu} 
+ \frac{\partial(x,\nabla\tilde u)}{\partial\nu}\Delta\tilde u\right)d\sigma
\end{split}
\end{equation}
for all $r > 0$. Combining (\ref{Sec6Eqt5}), (\ref{Sec6Eqt6}), and (\ref{Sec6Eqt7}), letting 
$r \to 0$, we then get that
\begin{equation}\label{Sec6ConvEqtConcl}
\mu_\alpha^{-\frac{n-4}{2}}R_\alpha \to K(u_\infty)
\end{equation}
as $\alpha\to+\infty$, where $K(u_\infty) = (n-2)(n-4)^2\omega_{n-1}A\hat\varphi(0)$. We have that $A > 0$ and we know that 
$\hat\varphi(0) > 0$ if $u_\infty\not\equiv 0$. It follows that $K(u_\infty) > 0$ if $u_\infty \not\equiv 0$.
By combining (\ref{PohTypeIdentity2})--(\ref{Sec6Eqt4}), and (\ref{Sec6ConvEqtConcl}), we can write that
\begin{equation}\label{Sec6ConclPohP1}
\begin{split}
&b_\alpha\int_{B_0(\delta\sqrt{\mu_\alpha})}\varphi^{\frac{4}{n-4}}\vert\nabla\hat u_\alpha\vert^2dx 
- \int_{B_0(\delta\sqrt{\mu_\alpha})}\varphi^{\frac{8}{n-4}}A_g(\nabla\hat u_\alpha,\nabla\hat u_\alpha)dx\\
&= o\left(\int_{B_0(\delta\sqrt{\mu_\alpha})}\vert\nabla\hat u_\alpha\vert^2dx\right) 
+ \bigl(K(u_\infty) + o(1)\bigr)\mu_\alpha^{\frac{n-4}{2}}\hskip.1cm ,
\end{split}
\end{equation}
where $o(1) \to 0$ as $\alpha \to +\infty$. The norm of $\nabla\hat u_\alpha$ in the first 
term of (\ref{Sec6ConclPohP1}) is with respect to the Euclidean metric $\hat g = \xi$. Noting that 
$\vert\nabla u\vert_{\hat g}^2 = \varphi^{4/(n-4)}\vert\nabla u\vert_g^2$, it follows from 
(\ref{Sec6ConclPohP1}) that
\begin{eqnarray*}
&&\int_{B_0(\delta\sqrt{\mu_\alpha})}\varphi^{\frac{8}{n-4}}
\bigl(A_g-b_\alpha g\bigr)(\nabla\hat u_\alpha,\nabla\hat u_\alpha)dx\\
&&= o\left(\int_{B_0(\delta\sqrt{\mu_\alpha})}\vert\nabla\hat u_\alpha\vert^2dx\right) 
- \bigl(K(u_\infty) + o(1)\bigr)\mu_\alpha^{\frac{n-4}{2}}
\end{eqnarray*}
an equation from which we easily get with \eqref{EqtLem7.2} that 
\begin{equation}\label{Sec6EqtConclArgProof2}
\begin{split}
&\int_{B_{x_\alpha}(\delta\sqrt{\mu_\alpha})}
\bigl(A_g-b_\alpha g\bigr)(\nabla u_\alpha,\nabla u_\alpha)dv_g\\
&= o\left(\int_{B_0(\delta\sqrt{\mu_\alpha})}\vert\nabla u_\alpha\vert^2dv_g\right) 
- \bigl(K(u_\infty) + o(1)\bigr)\mu_\alpha^{\frac{n-4}{2}}\hskip.1cm .
\end{split}
\end{equation}
This ends the proof of the lemma.
\end{proof}

Thanks to Lemma \ref{LemPrf1}, and to the estimates in Section \ref{PtEst}, we can prove Theorem \ref{LimitProfileThm}.

\begin{proof}[Proof of Theorem \ref{LimitProfileThm}] By Lemma  \ref{LemPrf1} it suffices to prove that when $n = 5, 6, 7$,
\begin{equation}\label{Thm1EstPf}
\int_{B_{x_\alpha}(\delta\sqrt{\mu_\alpha})}\vert\nabla u_\alpha\vert^2dv_g = o(\mu_\alpha^{\frac{n-4}{2}})
\hskip.1cm ,
\end{equation}
where $\delta > 0$ is as in Lemma \ref{LemPrf1}. For $\alpha \gg 1$ sufficiently large, we write that
\begin{equation}\label{Eqt1PfThm1}
\begin{split}
\int_{B_{x_\alpha}(\delta\sqrt{\mu_\alpha})}\vert\nabla u_\alpha\vert^2dv_g 
&\le \sum_{i=1}^k\int_{B_{x_{i,\alpha}}(R\mu_\alpha)}\vert\nabla u_\alpha\vert^2dv_g\\
&+ \int_{B_{x_\alpha}(\delta\sqrt{\mu_\alpha})\backslash\bigcup_{i=1}^kB_{x_{i,\alpha}}(R\mu_\alpha)}\vert\nabla u_\alpha\vert^2dv_g
\hskip.1cm ,
\end{split}
\end{equation}
where $R > 0$ is as in Proposition \ref{SharpPtEst}. By the embedding $H^2 \subset H^{1,2^\star}$, where 
$2^\star = \frac{2n}{n-2}$, the functions $\vert\nabla u_\alpha\vert$ are bounded in $L^{2^\star}$. Using 
H\"older's inequalities it follows that for any $i$,
\begin{equation}\label{Eqt2PfThm1}
\int_{B_{x_{i,\alpha}}(R\mu_\alpha)}\vert\nabla u_\alpha\vert^2dv_g 
\le C \mu_\alpha^{n(1-\frac{2}{2^\star})} = C\mu_\alpha^2
\hskip.1cm .
\end{equation}
There holds that $\mu_\alpha^2 = o(\mu_\alpha^{\frac{n-4}{2}})$ when $n = 5, 6, 7$. Independently, thanks to 
Proposition \ref{SharpPtEst},  we can write that
\begin{equation}\label{Eqt3PfThm1}
\begin{split}
&\int_{B_{x_\alpha}(\delta\sqrt{\mu_\alpha})\backslash\bigcup_{i=1}^kB_{x_{i,\alpha}}(R\mu_\alpha)}\vert\nabla u_\alpha\vert^2dv_g\\
&\le C\mu_\alpha^{\frac{n}{2}} + C\mu_\alpha^{n-4}
\sum_{i=1}^k\int_{B_{x_\alpha}(\delta\sqrt{\mu_\alpha})\backslash B_{x_{i,\alpha}}(R\mu_\alpha)}d_g(x_{i,\alpha},\cdot)^{6-2n}dv_g
\hskip.1cm .
\end{split}
\end{equation}
There holds that
\begin{equation}\label{Eqt4PfThm1}
\int_{B_{x_\alpha}(\delta\sqrt{\mu_\alpha})\backslash B_{x_{i,\alpha}}(R\mu_\alpha)}d_g(x_{i,\alpha},\cdot)^{6-2n}dv_g \le C S_\alpha
\end{equation}
for all $\alpha$, where $S_\alpha = 1$ when $n = 5$, $S_\alpha = \ln\frac{1}{\mu_\alpha}$ when $n = 6$, and 
$S_\alpha = \frac{1}{\mu_\alpha}$ when $n \ge 7$. 
Combining \eqref{Eqt1PfThm1}--\eqref{Eqt4PfThm1} we get \eqref{Thm1EstPf}. This ends the proof of Theorem \ref{LimitProfileThm}.
\end{proof}

\section{Trace estimates}\label{ProofTraceEst}

We prove trace estimates in this section. Such estimates are required to prove Theorem \ref{StabilityThm}. As in Section \ref{PtEst} we do not need to assume here that 
$g$ is conformally flat. We let $(M,g)$ be a compact Riemannian manifold  
and let $(b_\alpha)_\alpha$ and $(c_\alpha)_\alpha$ be converging sequences 
of real numbers with limits $b$ and $c$ as $\alpha \to \infty$, where $c - \frac{b^2}{4} < 0$. We let also $(u_\alpha)_\alpha$ be a bounded sequence in $H^2$ of 
positive nontrivial solutions of \eqref{PertCritEqt} satisfying \eqref{BlowUpAssumpt}. We aim at proving that if $A$ is a smooth $(2,0)$-tensor field, the 
the integral of $A(\nabla u_\alpha,\nabla u_\alpha)$ around the maximum blow-up point $x_\alpha$ behaves like the trace of $A$ at $x_\infty$ times $\mu_\alpha^2$, where 
$x_\alpha \to x_\infty$ as $\alpha \to +\infty$. 
In what follows we define 
$\tilde I_1$ and $\tilde I_2$ to be the subsets of $\{1,\dots,k\}$ given by
\begin{equation}\label{DefSubsetsInd}
\begin{split}
&\tilde I_1 = \Bigl\{i = 1,\dots,k~\hbox{s.t.}~d_g(x_{i,\alpha},x_\alpha) = o(1)\Bigr\}\hskip.1cm ,\hskip.1cm\hbox{and}\\
&\tilde I_2 = \Bigl\{i = 1,\dots,k~\hbox{s.t.}~d_g(x_{i,\alpha},x_\alpha) = o(\sqrt{\mu_\alpha})\Bigr\}\hskip.1cm ,
\end{split}
\end{equation}
where the $x_{i,\alpha}$'s and $k$ are given by the decomposition \eqref{EqtBlowUp1}, $\mu_\alpha$ is as in \eqref{DefMuAlpha}, and $x_\alpha$ is the blow-up point associated with 
$\mu_\alpha$. Namely, assuming that, up to a subsequence, $\mu_\alpha = \mu_{i_0,\alpha}$ for some $i_0$ and all $\alpha$, then 
$x_\alpha = x_{i_0,\alpha}$. 

\begin{prop}\label{TraceEst} Let $(M,g)$ be a smooth compact Riemannian manifold of dimension $n \ge 7$, and $b, c > 0$ be positive real 
numbers such that $c - \frac{b^2}{4} < 0$. Let also $(b_\alpha)_\alpha$ and $(c_\alpha)_\alpha$ be converging sequences 
of real numbers with limits $b$ and $c$ as $\alpha \to +\infty$, and $(u_\alpha)_\alpha$ be a bounded sequence in $H^2$ of 
positive nontrivial solutions of \eqref{PertCritEqt} satisfying \eqref{BlowUpAssumpt}. Let $A$ be a smooth $(2,0)$-tensor field. Let $\delta > 0$ be 
such that $d_g(x_{i,\alpha},x_\alpha) \ge 2\delta\sqrt{\mu_\alpha}$ for all $\alpha$ and all $i \not\in \tilde I_2$, where $\mu_\alpha$ is as in \eqref{DefMuAlpha}, 
$x_\alpha$ is the blow-up point associated with 
$\mu_\alpha$, and $\tilde I_2$ is as in \eqref{DefSubsetsInd}. Then there exists $\beta > 0$ 
such that, up to a subsequence,
\begin{equation}\label{TraceEqtPropEqt1}
\lim_{\alpha\to +\infty}\frac{1}{\mu_\alpha^2}
\int_{B_{x_\alpha}(\delta\sqrt{\mu_\alpha})}A(\nabla u_\alpha,\nabla u_\alpha)dv_g = \beta \hbox{Tr}_g(A)(x_\infty)
\hskip.1cm ,
\end{equation}
where $x_\infty$ is the limit of the $x_\alpha$'s. 
Similarly, if $u_\infty \equiv 0$, and $\delta > 0$ is 
such that $d_g(x_{i,\alpha},x_\alpha) \ge 2\delta$ for all $\alpha$ and all $i \not\in \tilde I_1$, where $x_\alpha$ is the blow-up point associated with 
$\mu_\alpha$, and $\tilde I_1$ is as in \eqref{DefSubsetsInd}, then
\begin{equation}\label{TraceEqtPropEqt2}
\lim_{\alpha\to +\infty}\frac{1}{\mu_\alpha^2}
\int_{B_{x_\alpha}(\delta)}A(\nabla u_\alpha,\nabla u_\alpha)dv_g = \beta \hbox{Tr}_g(A)(x_\infty)
\end{equation}
for some $\beta > 0$, where, here again, $x_\infty$ is the limit of the $x_\alpha$'s.
\end{prop}

\begin{proof}[Proof of Proposition \ref{TraceEst}]  Let $I^\prime$ be the subset of $\bigl\{1,\dots,k\bigr\}$ consisting of the $i$'s such that $\mu_\alpha = O(\mu_{i,\alpha})$ and, 
for $i$ given, let $\hat I_i$ be the subset of $\bigl\{1,\dots,k\bigr\}$ 
consisting of the $j$'s such that $d_g(x_{i,\alpha},x_{j,\alpha}) = O\left(\mu_{i,\alpha}\right)$. Given $R > 0$ we define $A_{i,\alpha,R}$ to be the annuli type sets
$$A_{i,\alpha,R} = B_{x_{i,\alpha}(R\mu_{i,\alpha})}\backslash \bigcup_{j \in \hat I_i}B_{x_{j,\alpha}}\left(\frac{1}{R}\mu_{i,\alpha}\right)
\hskip.1cm .$$
We claim that for any sequences $(\Omega_\alpha)_\alpha$ of domains in $M$,
\begin{equation}\label{ExteriorEstClaim}
\begin{split}
&\int_{\Omega_\alpha\backslash\bigcup_{i \in I^\prime}A_{i,\alpha,R}}\vert\nabla u_\alpha\vert^2dv_g \le 2\int_{\Omega_\alpha}\vert\nabla u_\infty\vert^2dv_g 
+ o\left(\hbox{Vol}_g(\Omega_\alpha)^{\frac{2}{n}}\right) + \varepsilon_R\mu_\alpha^2\hskip.1cm ,\\
&\int_{\Omega_\alpha\backslash\bigcup_{i=1}^kB_{x_{i,\alpha}}(R\mu_\alpha)}\vert\nabla u_\alpha\vert^2dv_g \le O\left(\hbox{Vol}_g(\Omega_\alpha)\right) 
+ \varepsilon_R\mu_\alpha^2
\end{split}
\end{equation}
for all $\alpha$, where $\varepsilon_R \to 0$ as $R \to +\infty$. First we prove \eqref{ExteriorEstClaim}, then we prove 
\eqref{TraceEqtPropEqt1} and at last we prove \eqref{TraceEqtPropEqt2}.

\medskip (1) {\it Proof of the first estimate in \eqref{ExteriorEstClaim}}. We use the Sobolev 
decomposition \eqref{EqtBlowUp1}. Thanks to \eqref{EqtBlowUp1}, by the Sobolev embedding $H^2 \subset H^{1,2^\star}$, where $2^\star = \frac{2n}{n-2}$, by H\"oder's inequality, 
and since $n \ge 7$ so that there holds $\mu_\alpha^{n-4} = o(\mu_\alpha^2)$, we can write that
\begin{equation}\label{SectTraceEqt1}
\begin{split}
&\int_{\Omega_\alpha\backslash\bigcup_{i \in I^\prime}A_{i,\alpha,R}}\vert\nabla u_\alpha\vert^2dv_g 
\le 2\int_{\Omega_\alpha}\vert\nabla u_\infty\vert^2dv_g\\
&\hskip.4cm + \sum_{j=1}^k\int_{\Omega_\alpha\backslash\bigcup_{i \in I^\prime}A_{i,\alpha,R}}\vert\nabla B_\alpha^j\vert^2dv_g
+ o\left(\hbox{Vol}_g(\Omega_\alpha)^{\frac{2}{n}}\right)\hskip.1cm .
\end{split}
\end{equation}
Independently, for any $j$,
\begin{equation}\label{SectTraceEqt2}
\int_{\Omega_\alpha\backslash\bigcup_{i \in I^\prime}A_{i,\alpha,R}}\vert\nabla B_\alpha^j\vert^2dv_g
\le C\mu_{j,\alpha}^2\int_{\mathbb{R}^n\backslash \frac{1}{\mu_{j,\alpha}}K_\alpha}\frac{\vert x\vert^2}{(1 + \vert x\vert^2)^{n-2}}dx + o(\mu_\alpha^2)
\hskip.1cm ,
\end{equation}
where $K_\alpha = \exp_{x_{j,\alpha}}^{-1}\left(\bigcup_{i \in I^\prime}A_{i,\alpha,R}\right)$. In case $j \not\in I^\prime$, then $\mu_{j,\alpha} = o(\mu_\alpha^2)$ and 
\begin{equation}\label{SectTraceEqt3}
\int_{\Omega_\alpha\backslash\bigcup_{i \in I^\prime}A_{i,\alpha,R}}\vert\nabla B_\alpha^j\vert^2dv_g
= o(\mu_\alpha^2)\hskip.1cm .
\end{equation}
In case $j \in I^\prime$, then 
\begin{equation}\label{SectTraceEqt4}
\begin{split}
&\int_{\Omega_\alpha\backslash\bigcup_{i \in I^\prime}A_{i,\alpha,R}}\vert\nabla B_\alpha^j\vert^2dv_g\\
&\le C\mu_{j,\alpha}^2\int_{\mathbb{R}^n\backslash \frac{1}{\mu_{j,\alpha}}\exp_{x_{j,\alpha}}^{-1}(A_{j,\alpha,R})}\frac{\vert x\vert^2}{(1 + \vert x\vert^2)^{n-2}}dx + o(\mu_\alpha^2)\\
&\le C\mu_{j,\alpha}^2\int_{\mathbb{R}^n\backslash K_R}\frac{\vert x\vert^2}{(1 + \vert x\vert^2)^{n-2}}dx + o(\mu_\alpha^2)
\end{split}
\end{equation}
where $K_R = B_0(R)\backslash\bigcup_{i\in \hat I_j}B_{y_i}(\frac{2}{R})$ and 
$y_i $ is the limit of the $\frac{1}{\mu_{j,\alpha}}\exp_{x_{j,\alpha}}^{-1}(x_{i,\alpha})$'s. The first estimate in 
\eqref{ExteriorEstClaim} clearly follows from \eqref{SectTraceEqt1}--\eqref{SectTraceEqt4}. 

\medskip (2) {\it Proof of the second estimate in \eqref{ExteriorEstClaim}}. Here we use 
Proposition \ref{SharpPtEst}. By \eqref{EqtMnPrp} we can write that 
\begin{equation}\label{SectTraceEqt5}
\begin{split}
&\int_{\Omega_\alpha\backslash\bigcup_{i=1}^kB_{x_{i,\alpha}}(R\mu_\alpha)}\vert\nabla u_\alpha\vert^2dv_g\\
&\le C\int_{\Omega_\alpha\backslash\bigcup_{i=1}^kB_{x_{i,\alpha}}(R\mu_\alpha)}\left(1+\mu_\alpha^{n-4}r_\alpha^{6-2n}\right)dv_g\\
&\le C\hbox{Vol}_g(\Omega_\alpha) + 
C\mu_\alpha^{n-4}\sum_{j=1}^k\int_{\Omega_\alpha\backslash B_{x_{j,\alpha}}(R\mu_\alpha)}d_g(x_{j,\alpha},x)^{6-2n}dv_g
\hskip.1cm ,
\end{split}
\end{equation}
and there holds that
\begin{equation}\label{SectTraceEqt6}
\int_{\Omega_\alpha\backslash B_{x_{j,\alpha}}(R\mu_\alpha)}d_g(x_{j,\alpha},x)^{6-2n}dv_g
\le C_1 + C_2\mu_\alpha^{6-n}\int_{\mathbb{R}^n\backslash B_0(R)}\vert x\vert^{6-2n}dx\hskip.1cm .
\end{equation}
Since $n \ge 7$, the second estimate in \eqref{ExteriorEstClaim} follows from \eqref{SectTraceEqt5} and \eqref{SectTraceEqt6}. This proves 
\eqref{ExteriorEstClaim}. 

\medskip (3) {\it Proof of \eqref{TraceEqtPropEqt1}}. Let $K^1_\alpha = \bigcup_{i\in I^\prime}A_{i,\alpha,R}$ and 
$K^2_\alpha = \bigcup_{i=1}^k B_{x_{i,\alpha}}(R\mu_\alpha)$. By the structure equation, \eqref{EqtBlowUp3}, $A_{i,\alpha,R}\bigcap A_{j,\alpha,R} = \emptyset$ 
for all $i \not= j$ in $I^\prime$. We start by writing that
\begin{equation}\label{TraceEstPfEqt1}
\begin{split}
&\int_{B_{x_\alpha}(\delta\sqrt{\mu_\alpha})}A(\nabla u_\alpha,\nabla u_\alpha)dv_g\\
&= \int_{B_{x_\alpha}(\delta\sqrt{\mu_\alpha})\bigcap K^1_\alpha}A(\nabla u_\alpha,\nabla u_\alpha)dv_g 
+ \int_{B_{x_\alpha}(\delta\sqrt{\mu_\alpha})\backslash K^1_\alpha}A(\nabla u_\alpha,\nabla u_\alpha)dv_g\\
&= \sum_{i \in J}\int_{A_{i,\alpha,R}}A(\nabla u_\alpha,\nabla u_\alpha)dv_g 
+ \int_{B_{x_\alpha}(\delta\sqrt{\mu_\alpha})\backslash K^1_\alpha}A(\nabla u_\alpha,\nabla u_\alpha)dv_g\hskip.1cm ,
\end{split}
\end{equation}
where $J = I^\prime\bigcap\tilde I_2$ and $\tilde I_2$ is as in \eqref{DefSubsetsInd}, and that
\begin{equation}\label{TraceEstPfEqt2}
\begin{split}
&\left\vert\int_{B_{x_\alpha}(\delta\sqrt{\mu_\alpha})\backslash K^1_\alpha}A(\nabla u_\alpha,\nabla u_\alpha)dv_g\right\vert\\
&\le \left\vert\int_{B_{x_\alpha}(\delta\sqrt{\mu_\alpha})\backslash K^2_\alpha}A(\nabla u_\alpha,\nabla u_\alpha)dv_g\right\vert 
+ \left\vert\int_{K^2_\alpha\backslash K^1_\alpha}A(\nabla u_\alpha,\nabla u_\alpha)dv_g\right\vert\hskip.1cm .
\end{split}
\end{equation}
By \eqref{ExteriorEstClaim},
\begin{equation}\label{TraceEstPfEqt3}
\begin{split}
&\left\vert\int_{B_{x_\alpha}(\delta\sqrt{\mu_\alpha})\backslash K^2_\alpha}A(\nabla u_\alpha,\nabla u_\alpha)dv_g\right\vert 
\le o(\mu_\alpha^2) + C\varepsilon_R\mu_\alpha^2
\hskip.1cm ,\\
&\left\vert\int_{K^2_\alpha\backslash K^1_\alpha}A(\nabla u_\alpha,\nabla u_\alpha)dv_g\right\vert 
\le o(\mu_\alpha^2) + C\varepsilon_R\mu_\alpha^2\hskip.1cm ,
\end{split}
\end{equation}
where $\varepsilon_R \to 0$ as $R\to +\infty$. Combining \eqref{TraceEstPfEqt1}--\eqref{TraceEstPfEqt3}, it follows that 
\begin{equation}\label{TraceEstPfEq4}
\begin{split}
\int_{B_{x_\alpha}(\delta\sqrt{\mu_\alpha})}A(\nabla u_\alpha,\nabla u_\alpha)dv_g
&= \sum_{i \in J}\int_{A_{i,\alpha,R}}A(\nabla u_\alpha,\nabla u_\alpha)dv_g\\
&+ o(\mu_\alpha^2) + \varepsilon_R(\alpha)\mu_\alpha^2\hskip.1cm ,
\end{split}
\end{equation}
where $\lim_{R\to +\infty}\lim_{\alpha\to +\infty}\varepsilon_R(\alpha) = 0$. We fix $i \in J$ and define $g_\alpha$ to be the metric in Euclidean 
space given by $g_\alpha(x) = \bigl(\exp_{x_{i,\alpha}}^\star g\bigr)(\mu_{i,\alpha}x)$. For $j \in \hat I_i$, we let also $a_{j,\alpha}$ be the point in $\mathbb{R}^n$ given by 
$a_{j,\alpha} = \mu_{i,\alpha}^{-1}\exp_{x_{i,\alpha}}^{-1}(x_{j,\alpha})$. Since $j \in \hat I_i$ there holds that, up to a subsequence, 
$a_{j,\alpha} \to a_j$ in $\mathbb{R}^n$. We define $\tilde u_{i,\alpha}$ to be the function defined in the Euclidean space by $\tilde u_\alpha = 
R_{x_{i,\alpha}}^{\mu_{i,\alpha}}u_\alpha$, where the $R_x^\mu$ action is as in \eqref{GroundStateShape}. In other words,
$$\tilde u_\alpha(x) = \mu_{i,\alpha}^{\frac{n-4}{2}}u_\alpha\left(\exp_{x_{i,\alpha}}(\mu_{i,\alpha}x)\right)\hskip.1cm .$$
Then
\begin{equation}\label{TraceEstPfEq5}
\int_{A_{i,\alpha,R}}A(\nabla u_\alpha,\nabla u_\alpha)dv_g = \mu_{i,\alpha}^2
\int_{B_0(R)\backslash W_\alpha}A_\alpha(\nabla\tilde u_\alpha,\nabla\tilde u_\alpha)dv_{g_\alpha}\hskip.1cm ,
\end{equation}
where $A_\alpha(x) = \bigl(\exp_{x_{i,\alpha}}^\star A\bigr)(\mu_{i,\alpha}x)$, 
$W_\alpha = \bigcup_{j \in \hat I_i}\tilde B_{a_{j,\alpha}}(\frac{1}{R})$, and $\tilde B_{a_{j,\alpha}}(\frac{1}{R})$ is the ball of center 
$a_{j,\alpha}$ and radius $1/R$ with respect to $g_\alpha$. Since $g_\alpha \to \xi$ in the $C^4$-topology, where $\xi$ is the Euclidean metric, 
it follows from Lemma \ref{Lem2} and \eqref{TraceEstPfEq5} that
\begin{equation}\label{TraceEstPfEq6}
\begin{split}
&\int_{A_{i,\alpha,R}}A(\nabla u_\alpha,\nabla u_\alpha)dv_g\\
&= \mu_{i,\alpha}^2 \int_{B_0(R)\backslash\bigcup_{i \in \hat I_i}B_{a_j}(\frac{1}{R})}\hat A_0(\nabla B,\nabla B)dx + o(\mu_{i,\alpha}^2)\\
&= \mu_{i,\alpha}^2 \int_{\mathbb{R}^n}\hat A_0(\nabla B,\nabla B)dx + o(\mu_{i,\alpha}^2) + \varepsilon_R\mu_\alpha^2\hskip.1cm ,
\end{split}
\end{equation}
where $\hat A_0 = \left(\exp_{x_\infty}^\star A\right)(0)$, $a_{j,\alpha} \to a_j$ 
as $\alpha \to +\infty$, $\varepsilon_R \to 0$ as $R \to +\infty$, and $B$ is as in \eqref{GroundStateShape0}. Since $B$ 
is radially symmetrical, we get from \eqref{TraceEstPfEq6} that
\begin{equation}\label{TraceEstPfEq7}
\begin{split}
&\int_{A_{i,\alpha,R}}A(\nabla u_\alpha,\nabla u_\alpha)dv_g\\
&= \mu_{i,\alpha}^2\frac{1}{n}\hbox{Tr}_g(A)(x_\infty)\int_{\mathbb{R}^n}\vert\nabla B\vert^2dx + o(\mu_{i,\alpha}^2) + \varepsilon_R\mu_\alpha^2\hskip.1cm ,
\end{split}
\end{equation}
and \eqref{TraceEqtPropEqt1} follows from \eqref{TraceEstPfEq4} and \eqref{TraceEstPfEq7} 
with $\beta = \left(\sum_{i\in J}\mu_i\right)\int_{\mathbb{R}^n}\vert\nabla B\vert^2dx$, where $\mu_i$ is the limit of $\frac{\mu_{i,\alpha}}{\mu_\alpha}$ as 
$\alpha \to +\infty$. Assuming $\mu_\alpha = \mu_{1,\alpha}$ for all $\alpha$, there holds that $1 \in J$ and $\beta > 0$. 
This ends the proof of  \eqref{TraceEqtPropEqt1}. 

\medskip (4) {\it Proof of \eqref{TraceEqtPropEqt2}}. We take advantage of $u_\infty \equiv 0$. 
By \eqref{EqtMnPrp} in Proposition \ref{SharpPtEst}, and since $n \ge 7$, we get that
\begin{equation}\label{TraceEstPfEq8}
\begin{split}
\int_{B_{x_\alpha}(\delta)\backslash K_\alpha^2}\vert\nabla u_\alpha\vert^2dv_g 
&\le C\mu_\alpha^{n-4}\int_{B_{x_\alpha}(\delta)\backslash K_\alpha^2}r_\alpha(x)^{6-2n}dv_g(x)\\
&\le C\mu_\alpha^{n-4}\sum_{i=1}^k\int_{B_{x_\alpha}(\delta)\backslash B_{x_{i,\alpha}}(R\mu_\alpha)}d_g(x_{i,\alpha},x_\alpha)^{6-2n}dv_g\\
&\le C\mu_\alpha^{n-4}\sum_{i=1}^k\left(1 + \mu_\alpha^{6-n}\int_{\mathbb{R}^n\backslash B_0(R)}\vert x\vert^{6-2n}dx\right)\\
&= o(\mu_\alpha^2) + \varepsilon_R\mu_\alpha^2\hskip.1cm ,
\end{split}
\end{equation}
where $\varepsilon_R \to 0$ as $R \to +\infty$. Then, we can write that
\begin{equation}\label{TraceEstPfEq9}
\begin{split}
&\int_{B_{x_\alpha}(\delta)}A(\nabla u_\alpha,\nabla u_\alpha)dv_g\\
&=  \sum_{i \in J^\prime}\int_{A_{i,\alpha,R}}A(\nabla u_\alpha,\nabla u_\alpha)dv_g 
+ \int_{B_{x_\alpha}(\delta)\backslash K^1_\alpha}A(\nabla u_\alpha,\nabla u_\alpha)dv_g
\hskip.1cm ,
\end{split}
\end{equation}
where $J^\prime = I^\prime\bigcap\tilde I_1$ and $\tilde I_1$ is as in \eqref{DefSubsetsInd}, while
\begin{equation}\label{TraceEstPfEq10}
\begin{split}
&\left\vert\int_{B_{x_\alpha}(\delta)\backslash K^1_\alpha}A(\nabla u_\alpha,\nabla u_\alpha)dv_g\right\vert\\
&\le \left\vert\int_{B_{x_\alpha}(\delta)\backslash K^2_\alpha}A(\nabla u_\alpha,\nabla u_\alpha)dv_g\right\vert
+ \left\vert\int_{K^2_\alpha\backslash K^1_\alpha}A(\nabla u_\alpha,\nabla u_\alpha)dv_g\right\vert
\hskip.1cm .
\end{split}
\end{equation}
By \eqref{ExteriorEstClaim} and \eqref{TraceEstPfEq8} we then get from 
\eqref{TraceEstPfEq9} and \eqref{TraceEstPfEq10} that
\begin{equation}\label{TraceEstPfEq11}
\int_{B_{x_\alpha}(\delta)}A(\nabla u_\alpha,\nabla u_\alpha)dv_g 
=  \sum_{i \in J^\prime}\int_{A_{i,\alpha,R}}A(\nabla u_\alpha,\nabla u_\alpha)dv_g + o(\mu_\alpha^2) 
+ \varepsilon_R\mu_\alpha^2\hskip.1cm ,
\end{equation}
where $\varepsilon_R \to 0$ as $R \to +\infty$, and \eqref{TraceEqtPropEqt2}
follow from \eqref{TraceEstPfEq7} and \eqref{TraceEstPfEq11}. This ends the proof of the proposition.
\end{proof}

\section{Proof of Theorem \ref{StabilityThm} when $n \ge 6$}\label{ProofTheorem2Part1}

We prove Theorem \ref{StabilityThm} by contradiction. We assume that $(M,g)$ is conformally flat of dimension $n \ge 6$. 
We let $(b_\alpha)_\alpha$ and $(c_\alpha)_\alpha$ be converging sequences 
of real numbers with limits $b$ and $c$ as $\alpha \to \infty$, and $(u_\alpha)_\alpha$ be a bounded sequence in $H^2$ of 
positive nontrivial solutions of \eqref{PertCritEqt} satisfying \eqref{BlowUpAssumpt}. We split the proof in the two cases 
$n = 6,7$ and $n \ge 8$.

\medskip First we assume $n = 6, 7$. By Theorem \ref{LimitProfileThm} we know that $u_\infty \equiv 0$ in 
\eqref{EqtBlowUp1}. Let $\mathcal{S} = \left\{x_1,\dots,x_N\right\}$ be the geometric blow-up set consisting 
of the limits of the $x_{i,\alpha}$'s, where the $x_{i,\alpha}$'s are as in Section \ref{PtEst}. Let $x_\alpha$ and $\mu_\alpha$ be as in 
Sections \ref{PtEst} and \ref{ProofThm1}, $\mu_\alpha$ being 
as in \eqref{DefMuAlpha}. We may assume $x_\alpha = x_{1,\alpha}$ for all $\alpha$. 
Given $x_i \in \mathcal{S}$, since $g$ is conformally flat, there exists (up to the assimilation of $x_i$ with $0$) a 
smooth positive function $\varphi > 0$ in a neighborhood $U$ of $x_i$ such that $\varphi^{4/(n-4)}\xi = g$ in $U = B_0(\delta_0)$, 
where $\xi$ is the Euclidean metric. We may also assume $U\cap\mathcal{S} =Ê\{x_i\}$. 
We define $\hat u_\alpha = \varphi u_\alpha$ and apply the Pohozaev identity \eqref{PohTypeIdentity} to $\eta_\delta\hat u_\alpha$ in $B_0(\delta)$ for 
$\delta \in (0,\delta_0)$, where $\eta_\delta(x) = \eta(\frac{2}{\delta}x)$ and $\eta$ is such that $\eta \equiv 1$ in $B_0(1)$ 
and $\eta \equiv 0$ in $\mathbb{R}^n\backslash B_0(4/3)$. By Hebey, Robert and Wen \cite{HebRobWen}, there holds that
\begin{equation}\label{ConcEstRelat}
\int_{B_0(\delta)\backslash B_0(\delta/2)}\vert\nabla^k\hat u_\alpha\vert^2dx = o(1)\int_M\vert\nabla u_\alpha\vert^2dx
\end{equation}
for all $k = 0,1,2$, where $o(1) \to 0$ as $\alpha \to +\infty$, and there also holds since $u_\infty \equiv 0$ that
\begin{equation}\label{ConcEstRelat2}
\frac{\int_{\mathcal{B}_\delta}\vert\nabla u_\alpha\vert^2dv_g}{\int_M\vert\nabla u_\alpha\vert^2dv_g} \to 1
\end{equation}
as $\alpha \to +\infty$, where $\mathcal{B}_\delta = \bigcup_{i=1}^NB_{x_i}(\delta)$. These estimates may be proved 
directly from Proposition \ref{SharpPtEst}. 
By \eqref{EuclEqt} and \eqref{ConcEstRelat}, following the computations in Hebey, Robert and Wen 
\cite{HebRobWen}, we get from the Pohozaev identity that
\begin{equation}\label{PohozConcl}
\left\vert\int_{\mathbb{R}^n}\eta^2\varphi^{\frac{8}{n-4}}\left(A_g-b_\alpha g\right)(\nabla\hat u_\alpha,\nabla\hat u_\alpha)dx\right\vert
\le C\left(\varepsilon_\delta + o(1)\right)\int_M\vert\nabla u_\alpha\vert^2dv_g\hskip.1cm ,
\end{equation}
where $C > 0$ is independent of $\alpha$ and $\delta$, $A_g$ is as in \eqref{DefAg}, and 
$\varepsilon_\delta$ can be made independent of $\alpha$ and such that 
$\varepsilon_\delta \to 0$ as $\delta \to 0$. 
When $b \not\in \mathcal{S}_w$, $A_g-b_\alpha g$ has a sign for $\alpha \gg 1$ sufficiently large. 
In particular, coming back to $M$, summing over $i = 1,\dots, N$, it follows from \eqref{ConcEstRelat}, \eqref{ConcEstRelat2} and 
\eqref{PohozConcl} that
\begin{equation}\label{PohozConcl2}
\int_M\vert\nabla u_\alpha\vert^2dv_g
\le C\varepsilon_\delta\int_M\vert\nabla u_\alpha\vert^2dv_g + o\left(\int_M\vert\nabla u_\alpha\vert^2dv_g\right)\hskip.1cm ,
\end{equation}
and we get a contradiction since $\varepsilon_\delta \to 0$ as $\delta \to 0$. This proves Theorem \ref{StabilityThm} when $n = 6$. When $n = 7$, we consider 
\eqref{PohozConcl} around $x_\alpha$, namely for $i = 1$. By \eqref{Eqt2PfThm1}, \eqref{ConcEstRelat2}, and Proposition \ref{SharpPtEst},
\begin{equation}\label{L2mu2Control}
\int_M\vert\nabla u_\alpha\vert^2dv_g = O(\mu_\alpha^2)\hskip.1cm .
\end{equation}
By \eqref{TraceEqtPropEqt2} in Proposition \ref{TraceEst}, \eqref{ConcEstRelat}, and \eqref{L2mu2Control}, we then get by letting $\alpha \to +\infty$ and $\delta \to 0$ 
in \eqref{PohozConcl} that 
$\frac{1}{n}\hbox{Tr}_g(A_g)(x_\infty) = b$, 
where $x_\infty$ is the limit of the $x_\alpha$'s. This proves Theorem \ref{StabilityThm} when $n = 7$.

\medskip Now we assume $n \ge 8$. Let $x_\alpha$ and $\mu_\alpha$ be as above. By \eqref{SharpPtEst} and Proposition \ref{TraceEst},
\begin{equation}\label{EstGradL2}
\int_{B_{x_\alpha}(\delta\sqrt{\mu_\alpha})}\vert\nabla u_\alpha\vert^2dv_g = O(\mu_\alpha^2)
\end{equation}
for all $\delta > 0$. Applying Lemma \ref{LemPrf1} and Proposition \ref{TraceEst}, it follows that
\begin{equation}\label{ConcEqtSec4Eqt1}
n\left(\frac{1}{n}\hbox{Tr}_g(A_g)(x_\infty) - b + o(1)\right)\mu_\alpha^2 = -\frac{1}{\beta}\left(K(u_\infty) + o(1)\right)\mu_\alpha^{\frac{n-4}{2}}
\hskip.1cm ,
\end{equation}
where $K(u_\infty)$ is as in Lemma \ref{LemPrf1}, $\beta > 0$ is as in Proposition \ref{TraceEst}, and $x_\alpha \to x_\infty$ as $\alpha \to +\infty$. Assuming that $n \ge 9$ and 
$b \not= \frac{1}{n}\hbox{Tr}_g(A_g)$ in $M$, the contradiction directly follows from \eqref{ConcEqtSec4Eqt1} since, in that case, 
$\mu_\alpha^{(n-4)/2} = o(\mu_\alpha^2)$. This proves Theorem \ref{StabilityThm} when $n \ge 9$. 
In case $n = 8$ we have that $\mu_\alpha^{(n-4)/2} = \mu_\alpha^2$, and if we assume that 
$b < \frac{1}{n}\hbox{Tr}_g(A_g)$ in $M$, then, again, we directly get a contradiction thanks to \eqref{ConcEqtSec4Eqt1} using 
the signs of the two terms in \eqref{ConcEqtSec4Eqt1}. This ends the proof of Theorem \ref{StabilityThm}.

\section{Proof of Theorem \ref{StabilityThm} when $n = 5$}\label{ProofTheorem2Part2}

We prove Theorem \ref{LimitProfileThm} in the $5$-dimensional case by contradiction. We assume that $(M,g)$ is conformally flat of dimension $n = 5$. 
We let $(b_\alpha)_\alpha$ and $(c_\alpha)_\alpha$ be converging sequences 
of real numbers with limits $b$ and $c$ as $\alpha \to \infty$, and $(u_\alpha)_\alpha$ be a bounded sequence in $H^2$ of 
positive nontrivial solutions of \eqref{PertCritEqt} satisfying \eqref{BlowUpAssumpt}. By Theorem \ref{LimitProfileThm} we know that $u_\infty \equiv 0$. We let 
$\mathcal{S}$ be the geometric blow-up set consisting of the limits of the $x_{i,\alpha}$'s as $\alpha \to +\infty$:
$\mathcal{S} = \bigl\{x_1,\dots,x_N\bigr\}$, 
where $N \le k$. In the case of clusters, $N < k$. We prove in what follows that there exist $\lambda_1,\dots,\lambda_N \ge 0$ 
such that $\sum_{i=1}^N\lambda_i = 1$ and such that
\begin{equation}\label{MainEqt5dim}
\lambda_i^2\mu_{x_i}(x_i) + \sum_{j\not= i}\lambda_i\lambda_jG(x_i,x_j) = 0
\end{equation}
for all $i = 1,\dots,N$, where $G$ is the Green's function of $\Delta_g^2+b\Delta_g+c$ and $\mu_x$ is 
its regular part as in \eqref{GreenFct}. When $c < b^2/4$, which is assumed here, $G$ is given by
$$G(x,y) = \int_MG_1(x,z)G_2(z,y)dv_g(z)\hskip.1cm ,$$
where $G_1$ (respectively $G_2$) is the Green's function of the second order Schr\"odinger operator $\Delta_g + d_1$ 
(respectively $\Delta_g + d_2$), and $d_1$, $d_2$ are as in \eqref{SplittCsts} with 
$b$ and $c$ in place of $b_\alpha$ and $c_\alpha$. Hence, $G > 0$ and Theorem \ref{StabilityThm} when $n = 5$ follows 
from \eqref{MainEqt5dim}. Note that \eqref{MainEqt5dim} reduces to $\lambda_i^2\mu_{x_i}(x_i) = 0$ in case $N = 1$, so that 
the positivity of the mass is required, in particular in the case of clusters. 

\medskip We prove \eqref{MainEqt5dim} in the sequel. By Theorem \ref{LimitProfileThm} 
and Proposition \ref{SharpPtEst}, splitting $M$ into the two 
subsets $\{r_\alpha \le R\mu_\alpha\}$ and $\{r_\alpha \ge R\mu_\alpha\}$, we easily get that there exists $C >0$ such that, up to 
a subsequence, $\int_Mu_\alpha^{2^\sharp-1}dv_g \le C\mu_\alpha^{1/2}$ for all $\alpha$. By Lemma \ref{Lem2} we then easily get that there 
exists $c > 0$ such that, up to a subsequence,
\begin{equation}\label{LNormSec5Eqt1}
\int_Mu_\alpha^{2^\sharp-1}dv_g = \bigl(c+o(1)\bigr)\mu_\alpha^{\frac{1}{2}}\hskip.1cm .
\end{equation}
Again by Theorem \ref{LimitProfileThm} and Proposition \ref{SharpPtEst}, thanks also to 
\eqref{LNormSec5Eqt1}, we get that for any compact subset $\Omega$ of $M\backslash\mathcal{S}$, 
\begin{equation}\label{LNormSec5Eqt2}
\frac{\int_\Omega u_\alpha^{2^\sharp-1}dv_g}{\int_Mu_\alpha^{2^\sharp-1}dv_g} = o(1)\hskip.1cm .
\end{equation}
In what follows we let $\delta_0 = \inf_{i\not= j}d_g(x_i,x_j)$. For $i = 1,\dots,N$, and $\delta \in (0,\delta_0)$, we define
\begin{equation}\label{DefLambdaiSec5}
\lambda_i = \lim_{\alpha\to +\infty}\frac{\int_{B_{x_i}(\delta)}u_\alpha^{2^\sharp-1}dv_g}{\int_Mu_\alpha^{2^\sharp-1}dv_g}\hskip.1cm .
\end{equation}
It follows from \eqref{LNormSec5Eqt2} that $\lambda_i$ does not depend on $\delta$ and that $\sum_i\lambda_i = 1$. Let $\tilde u_\alpha$ be 
given by
\begin{equation}\label{DefTildeuAlphaSec5}
\tilde u_\alpha = \frac{u_\alpha}{\int_Mu_\alpha^{2^\sharp-1}dv_g}\hskip.1cm .
\end{equation}
By \eqref{PertCritEqt} and \eqref{LNormSec5Eqt1} there holds that
$$\Delta_g^2\tilde u_\alpha + b_\alpha\Delta_g\tilde u_\alpha + c_\alpha\tilde u_\alpha = 
\tilde\mu_\alpha^4 \tilde u_\alpha^{2^\sharp-1}\hskip.1cm ,$$
where $\tilde\mu_\alpha = O(\mu_\alpha)$. By Proposition \ref{SharpPtEst} and \eqref{LNormSec5Eqt1} there also holds that for any 
compact subset $\Omega \subset M\backslash\mathcal{S}$ there exists $C_\Omega > 0$ such that $\tilde u_\alpha \le C_\Omega$ 
in $\Omega$. Then, by standard elliptic theory, there exists $\tilde u \in C^4(M\backslash\mathcal{S})$ such that $\tilde u_\alpha \to \tilde u$ 
in $C^4_{loc}(M\backslash\mathcal{S})$ as $\alpha \to +\infty$. By Green's representation formula and the estimates in 
\eqref{Green'sEstimate1st}, we get that $\tilde u$ expresses as the sum of the $\lambda_iG_{x_i}$'s, where $G_{x_i} = G(x_i,\cdot)$. Summarizing,
up to a subsequence,
\begin{equation}\label{ConTildeuAlphaSec5}
\tilde u_\alpha \to \sum_{i=1}^N\lambda_iG_{x_i}
\end{equation}
in $C^4_{loc}(M\backslash\mathcal{S})$ as $\alpha \to +\infty$, where the $\lambda_i$'s are as in \eqref{DefLambdaiSec5} 
and $\tilde u_\alpha$ is given by \eqref{DefTildeuAlphaSec5}.

\medskip Now we fix $i \in \bigl\{1,\dots,N\bigr\}$. Since $g$ is conformally flat, there exists (up to the assimilation of $x_i$ with $0$) a 
smooth positive function $\varphi > 0$ in a neighborhood $U$ of $x_i$ such that $\varphi^{4/(n-4)}\xi = g$ in $U = B_0(\delta_0)$, 
where $\xi$ is the Euclidean metric. We may also assume $U\cap\mathcal{S} =Ê\{x_i\}$. Define $\hat u_\alpha = \varphi u_\alpha$. Basic 
Riemannian estimates, going back to the equation for geodesics, yield
\begin{equation}\label{DistEqt}
d_g(0,x) = \vert x\vert \varphi(0)^{\frac{2}{n-4}} \left(1 + \frac{1}{n-4}\left(\frac{\nabla\varphi(0)}{\varphi(0)},x\right) + O(\vert x\vert^2)\right)\hskip.1cm ,
\end{equation}
where $(\cdot,\cdot)$ is the Euclidean scalar product. 
It follows from \eqref{GreenFct}, \eqref{ConTildeuAlphaSec5} and \eqref{DistEqt} that
\begin{equation}\label{EqtHiPrim}
\lim_{\alpha\to +\infty}\frac{\hat u_\alpha}{\int_Mu_\alpha^{2^\sharp-1}dv_g} = H_i
\end{equation}
in $C^4_{loc}(U\backslash\{0\})$ as $\alpha \to +\infty$, where
\begin{equation}\label{EqtHi}
H_i(x) = \frac{\lambda_i\varphi(x_i)^{-1}}{6\omega_4\vert x\vert} + \beta_i(x)
\end{equation}
in $U\backslash\{0\}$, $\beta_i \in C^{0,\theta}(U)$ for $0 < \theta < 1$, $\beta_i$ is smooth outside $0$, and
\begin{equation}\label{EqtBetai0}
\beta_i(0) = \left(\lambda_i\mu_{x_i}(x_i) + \sum_{j\not= i}\lambda_jG(x_i,x_j)\right)\varphi(x_i)
\hskip.1cm .
\end{equation}
By standard elliptic theory, following arguments as in Druet, Hebey and V\'etois \cite{DruHebVet}, there also holds that
\begin{equation}\label{ControlEqtBetai}
\lim_{r \to 0}\sup_{\vert x\vert = r}\sum_{k=1}^3\vert x\vert^k\vert\nabla^k\beta_i(x)\vert = 0\hskip.1cm .
\end{equation}
In order to prove \eqref{ControlEqtBetai} in our context we first note that by \eqref{EuclEqt}, $\beta_i$ satisfies an equation like 
\begin{equation}\label{EqtPDEBetai}
\Delta^2\beta_i + A^{kl}\partial^2_{kl}\beta_i + B^k\partial_k\beta_i + D\beta_i = f_i
\end{equation}
in $U\backslash\{0\}$, where the coefficients $A^{kl}$, $B^k$ and $D$ are smooth, and where 
$f_i$ is such that $\vert f_i(x)\vert \le C\vert x\vert^{-3}$ in $U\backslash\{0\}$. First, keeping in mind that we aim at  
proving \eqref{ControlEqtBetai}, we claim that there exists $C > 0$ such that
\begin{equation}\label{ControlEqtBetaiS1}
\sum_{k=1}^3\vert x\vert^k\vert\nabla^k\beta_i(x)\vert \le C
\end{equation}
in $U\backslash\{0\}$. We argue by contradiction. Suppose that there exists $(x_m)_m$ in $U\backslash\{0\}$ 
such that $\sum_{k=1}^3\vert x_m\vert^k\vert\nabla^k\beta_i(x_m)\vert \to +\infty$ as $m \to +\infty$. Since $\beta_i$ is smooth 
in $U\backslash\{0\}$, there holds that $x_m \to 0$ as $m\to +\infty$. Let $\beta_{i,m}(x) = \beta_i(\vert x_m\vert x)$. 
By \eqref{EqtPDEBetai}, thanks to standard elliptic theory, there exists $\beta \in C^4(\mathbb{R}^n\backslash\{0\})$ 
such that $\beta_{i,m}\to\beta$ in $C^3_{loc}(\mathbb{R}^n\backslash\{0\})$ as $m \to +\infty$ and $\Delta^2\beta = 0$ 
in $\mathbb{R}^n\backslash\{0\}$. We have that $\vert\beta\vert \le C$ in $\mathbb{R}^n\backslash\{0\}$ since 
$\beta_i \in C^{0,\theta}(U)$. Then
\begin{equation}\label{ComputaARgBetai}
\sum_{k=1}^3\vert x_m\vert^k\left\vert\nabla^k\beta_i(x_m)\right\vert = 
\sum_{k=1}^3\left\vert\nabla^k\beta_{i,m}\left(\frac{x_m}{\vert x_m\vert}\right)\right\vert \to \sum_{k=1}^3\left\vert\nabla^k\beta(y)\right\vert\hskip.1cm ,
\end{equation}
where $y$ is the limit of the points $\frac{x_m}{\vert x_m\vert}$ as $m \to +\infty$. A contradiction, and this proves 
\eqref{ControlEqtBetaiS1}. Now we prove \eqref{ControlEqtBetai}. Here again we argue by contradiction. We assume there exists 
$(x_m)_m$ in $U\backslash\{0\}$ such that
\begin{equation}\label{ControlEqtBetaiS1bis}
\sum_{k=1}^3\vert x_m\vert^k\vert\nabla^k\beta_i(x_m)\vert \ge C
\end{equation}
for all $m$ and some $C > 0$, and such that $x_m \to 0$ as $m \to +\infty$. We define $\beta_{i,m}$ as above. Then we get the 
existence of $\beta \in C^4(\mathbb{R}^n\backslash\{0\}$ such that $\beta_{i,m}\to\beta$ in $C^3_{loc}(\mathbb{R}^n\backslash\{0\})$ 
as $m \to +\infty$ and $\Delta^2\beta = 0$ 
in $\mathbb{R}^n\backslash\{0\}$. By \eqref{ControlEqtBetaiS1}, there holds that $\Delta^2\beta = 0$ in $\mathbb{R}^n$ in the sense of distributions 
and not only outside $0$. Then $\beta$ is smooth and, necessarily, see Adimurthi, Robert and Struwe \cite{AdiRobStr}, we get that 
$\beta \equiv C^{st}$ is a constant. Coming back to \eqref{ComputaARgBetai}, we get a contradiction with 
\eqref{ControlEqtBetaiS1bis}. This proves \eqref{ControlEqtBetai}. 

\medskip From now on, given $\delta \in (0,\delta_0)$, we define
\begin{equation}\label{EqtADelta}
\begin{split}
A_\delta &= -\frac{1}{2}
\int_{\partial B_0(\delta)}\left(H_i\frac{\partial\Delta H_i}{\partial\nu} - \frac{\partial H_i}{\partial\nu}\Delta H_i\right)d\sigma 
+ \frac{1}{2}\int_{\partial B_0(\delta)}(x,\nu)(\Delta H_i)^2d\sigma\\
&- \int_{\partial B_0(\delta)}(x,\nabla H_i)\frac{\partial\Delta H_i}{\partial\nu} d\sigma
+\int_{\partial B_0(\delta)}\frac{\partial(x,\nabla H_i)}{\partial\nu}\Delta H_id\sigma\hskip.1cm ,
\end{split}
\end{equation}
where $\nu$ is the unit outward normal to $\partial B_0(\delta)$ and $H_i$ is as in \eqref{EqtHiPrim}-\eqref{EqtHi}. By 
\eqref{EqtHi} and \eqref{ControlEqtBetai},
\begin{equation}\label{FirstEqtADelta}
\lim_{\delta\to 0}A_\delta = \frac{\lambda_i\varphi(x_i)^{-1}}{2}\beta_i(0)\hskip.1cm .
\end{equation}
Independently, applying the Pohozaev identity \eqref{PohTypeIdentity} to $\hat u_\alpha$ in $B_0(\delta)$, we get by 
\eqref{LNormSec5Eqt1} and \eqref{EqtHiPrim} that 
\begin{equation}\label{LastArg5dimEqt1}
\int_{B_0(\delta)}\left(x^k\partial_k\hat u_\alpha + \frac{1}{2}\hat u_\alpha\right)\Delta^2\hat u_\alpha dx = 
\left(cA_\delta + o(1)\right)\mu_\alpha\hskip.1cm .
\end{equation}
By Proposition \ref{SharpPtEst}, for any $k \in \left\{0,1,2\right\}$, 
\begin{equation}\label{EstRemTermsLastArg}
\int_{B_0(\delta)}\hat u_\alpha\vert\nabla^k\hat u_\alpha\vert dx \le \varepsilon_\delta(\alpha)\mu_\alpha
\hskip.2cm\hbox{and}\hskip.2cm 
\int_{B_0(\delta)}\vert\nabla\hat u_\alpha\vert^2dx \le \varepsilon_\delta(\alpha)\mu_\alpha
\hskip.1cm ,
\end{equation}
where $\lim_{\delta \to 0}\limsup_{\alpha\to +\infty}\varepsilon_\delta(\alpha) = 0$. By \eqref{EuclEqt} and \eqref{EstRemTermsLastArg}, 
integrating by parts, we get that
\begin{equation}\label{LastArg5dimEqt2}
\left\vert\int_{B_0(\delta)}\left(x^k\partial_k\hat u_\alpha + \frac{1}{2}\hat u_\alpha\right)\Delta^2\hat u_\alpha dx\right\vert 
\le \varepsilon_\delta(\alpha)\mu_\alpha\hskip.1cm ,
\end{equation}
where $\varepsilon_\delta(\alpha)$ is as above. 
Combining \eqref{LastArg5dimEqt1} and \eqref{LastArg5dimEqt2} it follows that $A_\delta \to 0$ as $\delta \to 0$. Coming back to \eqref{EqtBetai0} and 
\eqref{FirstEqtADelta}, this proves \eqref{MainEqt5dim}. As already mentioned, this also proves Theorem \ref{StabilityThm} when $n = 5$.

\medskip Theorem \ref{StabilityThm} has an interpretation in terms of phase stability of solitons 
for the fourth order Schr\"odinger equation
\begin{equation}\label{SchroEvol}
i\frac{\partial u}{\partial t} + \Delta_g^2u + \varepsilon\Delta_gu
= \vert u\vert^{2^\sharp-2}u
\hskip.1cm ,
\end{equation}
where $\varepsilon > 0$. Equations like \eqref{SchroEvol} have been introduced by 
Karpman \cite{Kar} and Karpman and Shagalov \cite{KarSha} 
to take into account the role of small fourth-order dispersion terms in the 
propagation of intense laser beams in a bulk 
medium with Kerr nonlinearity. Among other possible references they have been investigated since then 
(local well-posedness, global well-posedness, scattering) 
by Fibich, Ilan, and Papanicolaou \cite{FibIlaPap}, Guo and Wang \cite{GuoWan}, 
Hao, Hsiao, and Wang \cite{HaoHsiWan1,HaoHsiWan2}, Pausader \cite{Pau1,Pau2,Pau3}, 
Pausader and Shao \cite{PauSha}, and Segata \cite{Seg}. Solitons for 
\eqref{SchroEvol} can be written as $ue^{-i\omega t}$, where 
$u: M \to \mathbb{R}$ satisfies \eqref{CritEqt} with $b = \varepsilon$ and 
$c = \omega$. We assume here that $\omega > 0$. If 
\eqref{CritEqt} with $b = \varepsilon$ and 
$c = \omega$ is stable, then phase stability holds true for 
\eqref{SchroEvol} in the sense that for any sequence 
$u_\alpha e^{-i\omega_\alpha t}$ of solitons, with $\Vert u_\alpha\Vert_{H^2} \le \Lambda$ 
for some $\Lambda > 0$, if $\omega_\alpha \to \omega$ in $\mathbb{R}$, then, 
up to a subsequence, $u_\alpha \to u$ in $C^4$ and the sequence 
of solitons converges to another soliton. In other words, if \eqref{CritEqt} is table, then the sole convergence of the 
phase suffices to guarantee convergence of the solitons. A corollary to Theorem 
\ref{StabilityThm} is that phase stability holds true for \eqref{SchroEvol} when 
the scalar curvature of the background space is positive, $\varepsilon > 0$ is 
sufficiently small, and $\omega \in (0,\varepsilon)$, up to the addition of extra 
assumptions when $n = 5$ in order to apply Theorem \ref{PositMass}.

\section{Proof of Theorem \ref{PositMass}}\label{ProofPosMassThm}

First we prove that $\mu_x(x) \ge 0$ for all $x$. 
Let $P_0$ be the geometric Paneitz operator as in the left hand side of \eqref{ConfEinstEqt}, and $P_g = \Delta_g^2 + b\Delta_g + c$. Let also 
$G_0$ be the Green's function of $P_0$ and $G$ be the Green's function of $P_g$. We fix $x \in M$, and let $\tau_x: M\backslash\{x\} \to \mathbb{R}$ 
be the function such that
\begin{equation}\label{DefTau}
G(x,\cdot) = G_0(x,\cdot) + \tau_x(\cdot)
\end{equation}
in $M\backslash\{x\}$. When $n = 5$, $\tau_x$ extends continuously in $M$. Moreover, we have that 
$P_g\tau_x = -P_gG_0(x,\cdot) = (P_0-P_g)G_0(x,\cdot)$ in 
$M\backslash\{x\}$. Noting that
$$(P_0-P_g)G(x,\cdot) = O(d_g(x,\cdot)^{-3})\hskip.1cm ,$$
we actually have that $\tau_x \in H_4^p(M)\cap C^{0,\theta}(M)$ for all 
$p \in (1,5/3)$ and all $\theta \in (0,1)$, where $H_4^p$ is the Sobolev space of functions in $L^p$ with four derivatives in $L^p$. In particular, 
$$\tau_x(y) = \int_MG(y,\cdot)\left(P_0-P_g\right)G_0(x,\cdot)dv_g$$
for all $y \in M$, and noting that $H_4^p \subset H_2^{\frac{5p}{5-2p}}$ and $\frac{5p}{5-2p} > 2$ for $p$ close to $5/3$, we get that 
\begin{equation}\label{SpliitingComputtsMass}
\begin{split}
\tau_x(x) &=\int_MG_0(x,\cdot)\left(P_0-P_g\right)G_0(x,\cdot)dv_g + \int_M\tau_xP_g\tau_xdv_g\\
&= \int_M\bigl(A_g-bg\bigr)(\nabla G_0\left(x,\cdot),\nabla G_0(x,\cdot)\right)dv_g\\
&\hskip.2cm + \int_M\left(\frac{1}{2}Q_g-c\right)G_0(x,\cdot)^2dv_g 
+ \int_M\left((\Delta_g\tau_x)^2 + b\vert\nabla\tau_x\vert^2 + c\tau_x^2\right)dv_g\hskip.1cm .
\end{split}
\end{equation}
By assumption, $bg \le A_g$ and $c \le \frac{1}{2}Q_g$.  Hence $\tau_x \ge 0$ in $M$. Now we use the fact that $g$ is conformally flat. 
In particular, there exists $\varphi > 0$ such that $g = \varphi^4\tilde g$ and $\tilde g$ is flat around $x$. 
The Green functions $G_0$ and $\tilde G_0$ of $P_0$ and $\tilde P_0$, where $\tilde P_0$ 
is the geometric Paneitz operator with respect to $\tilde g$, are related by
\begin{equation}\label{GreenLastEqt1}
G_0(x,y) = \frac{\tilde G_0(x,y)}{\varphi(x)\varphi(y)}
\end{equation}
for all $x \not= y$. Independently,
\begin{equation}\label{GreenLastEqt2}
\tilde G_0(x,y) = \frac{1}{6\omega_4d_{\tilde g}(x,y)} + A + \alpha_x(y)\hskip.1cm ,
\end{equation}
where $\alpha_x$ is continuous and such that $\alpha_x(x) = 0$. Combining \eqref{GreenLastEqt1} and \eqref{GreenLastEqt2}, 
thanks to \eqref{DistEqt}, we get that 
$$G_0(x,y) = \frac{1}{6\omega_4d_g(x,y)} + A + \tilde\alpha_x(y)\hskip.1cm ,$$
where $\tilde\alpha_x$ is such that $\tilde\alpha_x(x) = 0$. Coming back to \eqref{GreenFct}, thanks to \eqref{DefTau}, we then get that
\begin{equation}\label{MassEquiv}
\mu_x(x) = A + \tau_x(x)\hskip.1cm .
\end{equation}
By Humbert and Raulot \cite{HumRau}, assuming the Yamabe invariant is positive, $P_0$ is coercive, and $G_0$ is positive, we have that 
$A > 0$ with equality if and only if 
$(M,g)$ is conformally diffeomorphic to the unit sphere. 
Since $\tau_x(x) \ge 0$, and $x$ is arbitrary, we proved that $\mu_x(x) \ge 0$ for all $x$, and that if $\mu_x(x) = 0$ for some $x$, then 
$(M,g)$ is conformally diffeomorphic to the unit sphere. 

\medskip We assume now that $\mu_x(x) = 0$ for some $x$. Then, by \eqref{MassEquiv}, $\tau_x(x) = 0$ and $A = 0$. In particular $(M,g)$ is 
conformally diffeomorphic to the unit sphere and by \eqref{SpliitingComputtsMass}, since $b, c > 0$, $bg \le A_g$ and 
$c \le \frac{1}{2}Q_g$, we get that $\tau_x\equiv 0$ and that
\begin{equation}\label{ConclEqtsMassLast}
\frac{1}{2}Q_g \equiv c\hskip.1cm\hbox{in}\hskip.1cm M\hskip.2cm\hbox{and}\hskip.2cm \left(A_g-bg\right)\left(\nabla G(x,\cdot),\nabla G(x,\cdot)\right) \equiv 0
\hskip.1cm\hbox{in}\hskip.1cm M\backslash\{x\}\hskip.1cm .
\end{equation}
By first equation in \eqref{ConclEqtsMassLast}, $Q_g$ is constant, and since $g$ is conformal to the round metric we get, see for instance 
Hebey and Robert \cite{HebRob} for the classification of all constant metrics, that $g$ has constant sectional curvature. 
In particular, $(M,g)$ is isometric to the $5$-sphere 
with a constant multiple of the round metric. Then we also get that $A_g \equiv kg$ for some 
constant $k$ and it follows from the second equation in \eqref{ConclEqtsMassLast} that necessarily $k = b$. In particular, 
$c \equiv \frac{1}{2}Q_g$ and $A_g \equiv bg$ in $M$. This ends the proof of Theorem \ref{PositMass}.

\end{document}